\documentclass[a4paper,final]{amsart}

\usepackage{showkeys}
\usepackage{fullpage}
\usepackage{mathtools, amssymb, amsthm, amsfonts, mathrsfs}
\usepackage{bm}
\usepackage{enumitem}

\usepackage{color, graphicx, here}

\usepackage{tikz}
\usetikzlibrary{intersections, calc, arrows, cd, positioning, patterns}
\usepackage[colorlinks,citecolor=darkgreen,linkcolor=red]{hyperref}
\definecolor{darkgreen}{rgb}{0.0, 0.6, 0.0}

\numberwithin{equation}{section}
\numberwithin{figure}{section}
\allowdisplaybreaks

\setenumerate{label={\upshape (\arabic*)}, nosep}
\setitemize{nosep}
\setdescription{nosep}
\newlist{enua}{enumerate}{1}
\setlist*[enua]{label={\upshape (\arabic*)}, nosep}
\newlist{enur}{enumerate}{1}
\setlist*[enur]{label={\upshape (\roman*)}, nosep}

\usepackage{cleveref}
\crefname{thm}{Theorem}{Theorems}
\crefname{thma}{Theorem}{Theorems}
\crefname{prp}{Proposition}{Propositions}
\crefname{lem}{Lemma}{Lemmas}
\crefname{cor}{Corollary}{Corollaries}
\crefname{dfn}{Definition}{Definitions}
\crefname{rmk}{Remark}{Remarks}
\crefname{fct}{Fact}{Facts}
\crefname{ex}{Example}{Examples}
\crefname{warn}{Warning}{Warnings}

\theoremstyle{plain}
\newtheorem{thm}{Theorem}[section]
\newtheorem{prp}[thm]{Proposition}
\newtheorem{lem}[thm]{Lemma}
\newtheorem{cor}[thm]{Corollary}
\newtheorem{fct}[thm]{Fact}

\newtheorem*{thm*}{Theorem}
\newtheorem*{prp*}{Propostion}
\newtheorem*{lem*}{Lemma}
\newtheorem*{fct*}{Fact}
\newtheorem{thma}{Theorem}

\theoremstyle{definition}
\newtheorem{dfn}[thm]{Definition}
\newtheorem{nt}[thm]{Notation}
\newtheorem{rmk}[thm]{Remark}
\newtheorem{ex}[thm]{Example}

\newtheorem*{dfn*}{Definition}
\newtheorem*{nt*}{Notattion}
\newtheorem*{rmk*}{Remark}
\newtheorem*{ex*}{Example}

\newcommand{\ass}{\mathrm{ass}}
\newcommand{\sub}{\mathrm{sub}}
\newcommand{\quot}{\mathrm{quot}}
\newcommand{\ext}{\mathrm{ext}}

\newcommand{\T}{\mathsf{T}}
\newcommand{\F}{\mathsf{F}}
\newcommand{\kk}{\kappa}

\DeclareMathOperator{\NZD}{\mathsf{NZD}}
\DeclareMathOperator{\Assh}{Assh}
\DeclareMathOperator{\Min}{Min}

\DeclareMathOperator{\tf}{\mathsf{tf}}
\DeclareMathOperator{\pure}{\mathsf{pure}}

\DeclareMathOperator{\cm}{\mathsf{cm}}

\newcommand{\ol}{\overline}

\newcommand{\wti}{\widetilde}

\newcommand{\la}{\langle}
\newcommand{\ra}{\rangle}
\newcommand{\iso}{\cong}

\newcommand{\imply}{\Rightarrow}
\newcommand{\equi}{\Leftrightarrow}

\newcommand{\xr}[1]{\xrightarrow{\, #1 \, }}
\newcommand{\inj}{\hookrightarrow}
\newcommand{\surj}{\twoheadrightarrow}

\newcommand{\isoto}{\xr{\iso}}

\newcommand{\bbP}{\mathbb{P}}

\newcommand{\bbZ}{\mathbb{Z}}

\newcommand{\bbk}{\Bbbk}

\newcommand{\calB}{\mathcal{B}}

\newcommand{\calL}{\mathcal{L}}

\newcommand{\LL}{\Lambda}
\newcommand{\GG}{\Gamma}

\DeclareMathOperator{\rk}{rk}

\DeclareMathOperator{\Supp}{Supp}

\newcommand{\gen}[1]{\la #1 \ra}

\newcommand{\catA}{\mathcal{A}}

\newcommand{\catC}{\mathcal{C}}

\newcommand{\catF}{\mathcal{F}}

\newcommand{\catS}{\mathcal{S}}
\newcommand{\catT}{\mathcal{T}}

\newcommand{\catX}{\mathcal{X}}
\newcommand{\catY}{\mathcal{Y}}

\newcommand{\id}{\mathsf{id}}

\DeclareMathOperator{\Hom}{Hom}

\DeclareMathOperator{\Ext}{Ext}

\DeclareMathOperator{\Ima}{Im}
\DeclareMathOperator{\Ker}{Ker}
\DeclareMathOperator{\Cok}{Cok}

\DeclareMathOperator{\Mod}{\mathsf{Mod}}
\DeclareMathOperator{\catmod}{\mathsf{mod}}

\DeclareMathOperator{\fl}{\mathsf{fl}}

\DeclareMathOperator{\add}{\mathsf{add}}

\DeclareMathOperator{\Spec}{Spec}

\DeclareMathOperator{\Ass}{Ass}

\DeclareMathOperator{\depth}{depth}

\newcommand{\mm}{\mathfrak{m}}

\newcommand{\pp}{\mathfrak{p}}
\newcommand{\qq}{\mathfrak{q}}

\DeclareMathOperator{\Qcoh}{\mathsf{Qcoh}}

\DeclareMathOperator{\coh}{\mathsf{coh}}
\DeclareMathOperator{\tor}{\mathsf{tor}}
\DeclareMathOperator{\vect}{\mathsf{vect}}

\newcommand{\shE}{\mathscr{E}}
\newcommand{\shF}{\mathscr{F}}
\newcommand{\shG}{\mathscr{G}}
\newcommand{\shH}{\mathscr{H}}
\newcommand{\shI}{\mathscr{I}}

\newcommand{\shK}{\mathscr{K}}
\newcommand{\shL}{\mathscr{L}}

\newcommand{\shO}{\mathscr{O}}

\newcommand{\shV}{\mathscr{V}}

\DeclareMathOperator{\shHom}{\mathscr{H}\textit{\kern -3pt om}\hspace{0.5pt}}

\newcommand{\arr}[1]{\arrow[{#1}]}
\newenvironment{bsmatrix}{\left[\begin{smallmatrix}}{\end{smallmatrix}\right]}

\tikzset{
  symbol/.style={
    draw=none,
    every to/.append style={
      edge node={node [sloped, allow upside down, auto=false]{$#1$}}}
  },
  whitev/.style={circle, fill=white, draw=black, inner sep=1.5pt, outer sep=0pt}
}

\makeatletter
\newcommand*{\da@rightarrow}{\mathchar"0\hexnumber@\symAMSa 4B }
\newcommand*{\da@leftarrow}{\mathchar"0\hexnumber@\symAMSa 4C }
\newcommand*{\xdashrightarrow}[2][]{%
  \mathrel{%
    \mathpalette{\da@xarrow{#1}{#2}{}\da@rightarrow{\,}{}}{}%
  }%
}
\newcommand{\xdashleftarrow}[2][]{%
  \mathrel{%
    \mathpalette{\da@xarrow{#1}{#2}\da@leftarrow{}{}{\,}}{}%
  }%
}
\newcommand*{\da@xarrow}[7]{%
  \sbox0{$\ifx#7\scriptstyle\scriptscriptstyle\else\scriptstyle\fi#5#1#6\m@th$}%
  \sbox2{$\ifx#7\scriptstyle\scriptscriptstyle\else\scriptstyle\fi#5#2#6\m@th$}%
  \sbox4{$#7\dabar@\m@th$}%
  \dimen@=\wd0 %
  \ifdim\wd2 >\dimen@
    \dimen@=\wd2 %
  \fi
  \count@=2 %
  \def\da@bars{\dabar@\dabar@}%
  \@whiledim\count@\wd4<\dimen@\do{%
    \advance\count@\@ne
    \expandafter\def\expandafter\da@bars\expandafter{%
      \da@bars
      \dabar@ 
    }%
  }%
  \mathrel{#3}%
  \mathrel{%
    \mathop{\da@bars}\limits
    \ifx\\#1\\%
    \else
      _{\copy0}%
    \fi
    \ifx\\#2\\%
    \else
      ^{\copy2}%
    \fi
  }%
  \mathrel{#4}%
}

\newcommand{\colim@}[2]{%
  \vtop{\m@th\ialign{##\cr
    \hfil$#1\operator@font colim$\hfil\cr
    \noalign{\nointerlineskip\kern1.5\ex@}#2\cr
    \noalign{\nointerlineskip\kern-\ex@}\cr}}%
}
\newcommand{\colim}{%
  \mathop{\mathpalette\colim@{\rightarrowfill@\scriptscriptstyle}}\nmlimits@
}
\newcommand{\plim}{%
  \mathop{\mathpalette\varlim@{\leftarrowfill@\scriptscriptstyle}}\nmlimits@
}
\makeatother
\begin{document}

\title{Classifying torsionfree classes of the category of coherent sheaves and their Serre subcategories}

\author{Shunya Saito}
\address{Graduate School of Mathematics, Nagoya University, Chikusa-ku, Nagoya. 464-8602, Japan}
\email{m19018i@math.nagoya-u.ac.jp}


\begin{abstract}
In this paper, we classify several subcategories of the category of coherent sheaves 
on a divisorial noetherian scheme 
(e.g.\ a quasi-projective scheme over a commutative noetherian ring).
More precisely, we classify the torsionfree (resp.\ torsion) classes
\emph{closed under tensoring with line bundles} 
by the subsets (resp.\ specialization-closed subsets) of the scheme,
which generalizes the classification of torsionfree (resp.\ torsion) classes
of the category of finitely generated modules over a commutative noetherian ring
by Takahashi (resp.\ Stanley--Wang).

Furthermore, we classify the Serre subcategories of a torsionfree class
(in the sense of Quillen's exact categories) by using the above classifications,
which gives a certain generalization of Gabriel's classification of Serre subcategories.
As explicit applications, 
we classify the Serre subcategories of
the category of maximal pure sheaves,
which are a natural generalization of vector bundles for reducible schemes,
on a reduced projective curve over a field,
and the category of maximal Cohen-Macaulay modules over a one-dimensional Cohen-Macaulay ring.
\end{abstract}

\maketitle
\tableofcontents
\section{Introduction}\label{s:Intro}
Classifying nice subcategories of an abelian category or a triangulated category is
quite an active subject that has been studied in representation theory of algebras
and algebraic geometry.
See \cref{dfn:several subcat} for the definitions of several subcategories
appearing in the following.
One of the most classical results is 
Gabriel's classification of Serre subcategories in \cite{Gabriel}.
He classified the Serre subcategories of 
the category $\coh X$ of coherent sheaves on a noetherian scheme $X$
by the specialization-closed subsets of $X$.
For the category $\catmod R$ of finitely generated modules over a commutative noetherian ring $R$,
more kinds of subcategories were classified so far.
Takahashi \cite{Takahashi} showed that wide subcategories of $\catmod R$ are Serre subcategories
and classified the torsionfree classes of $\catmod R$ 
by the subsets of the prime spectrum $\Spec R$ of $R$.
Stanley and Wang \cite{SW} proved that
torsion classes and narrow subcategories of $\catmod R$ are Serre subcategories,
which extends the first half of Takahashi's result.
The most general result in the direction of \cite{SW} is 
the result due to Iima, Matsui, Shimada and Takahashi \cite{IMST},
which asserts that tensor-ideal subcategories of $\catmod R$ 
closed under direct summands and extensions are Serre subcategories.
Recently, Enomoto and Sakai \cite{IE} introduced
a new class of subcategories called \emph{IE-closed} subcategories,
which is a large class including torsion classes, torsionfree classes and wide subcategories
and intensively studied it for the category of finitely generated modules 
over a (non-commutative) finite dimensional algebra.
For a commutative noetherian ring $R$,
Enomoto \cite{IE=torf} proved that IE-closed subcategories of $\catmod R$ are torsionfree classes,
which immediately yields that torsion classes and wide subcategories are Serre subcategories.

The first aim of this paper is to extend the classification results of 
the module category $\catmod R$ of a commutative noetherian ring described above
to the category $\coh X$ of coherent sheaves on a noetherian scheme.
A naive extension does not hold.
For example, there exists a torsion class of 
the category $\coh \bbP^1$ of coherent sheaves on the projective line 
which is not a Serre subcategory (see \cite[Example 5.5]{CS}).
To remedy this phenomenon, 
we focus on the subcategories \emph{closed under tensoring with line bundles}
(cf.\ \cref{dfn:L-closed}).
Extending various techniques for $\catmod R$
to $\coh X$ \emph{up to tensoring with line bundles},
we obtain the following results.
\begin{thma}[{$=$ \cref{thm:Takahashi for scheme}}]\label{thma:torf}
Let $X$ be a divisorial noetherian scheme
(e.g.\ a quasi-projective scheme over a commutative noetherian ring).
Then the assignments
\[
\catX \mapsto \Ass\catX :=\bigcup_{\shF\in\catX} \Ass \shF
\quad \text{and} \quad
\Phi \mapsto \coh^{\ass}_{\Phi} X:=\{\shF\in \coh X \mid \Ass \shF\subseteq \Phi\}
\]
give rise to mutually inverse bijections between the following sets:
\begin{itemize}
\item 
The set of torsionfree classes of $\coh X$ closed under tensoring with line bundles.
\item
The power set of $X$.
\end{itemize}
\end{thma}
Applying \cref{thma:torf} to an affine scheme $X=\Spec R$,
we immediately obtain Takahashi's classification of torsionfree classes of $\catmod R$
(see \cref{fct:Takahashi torf}).

\begin{thma}[{$=$ \cref{thm:Stanley-Wang for scheme}}]\label{thma:ICE}
Let $X$ be a divisorial noetherian scheme.
Then the following are equivalent for a subcategory $\catX$ of $\coh X$:
\begin{enur}
\item
$\catX$ is a Serre subcategory.
\item
$\catX$ is a torsion class closed under tensoring with line bundles.
\item
$\catX$ is a wide subcategory closed under tensoring with line bundles.
\item
$\catX$ is an ICE-closed subcategory closed under tensoring with line bundles.
\item
$\catX$ is a narrow subcategory closed under tensoring with line bundles.
\item
$\catX$ is a tensor-ideal subcategory closed under direct summands and extensions.
\end{enur}
\end{thma}
Applying \cref{thma:ICE} to an affine scheme $X=\Spec R$,
we immediately obtain the results of Stanley--Wang and Iima--Matsui--Shimada--Takahashi for $\catmod R$
(see \cref{fct:tor=Serre}).

\begin{thma}[{$=$ \cref{thm:IE for scheme}}]\label{thma:IE}
Let $X$ be a divisorial noetherian scheme.
The following are equivalent for a subcategory $\catX$ of $\coh X$ 
closed under tensoring with line bundles:
\begin{itemize}
\item
$\catX$ is a torsionfree class.
\item
$\catX$ is an IKE-closed subcategory.
\item
$\catX$ is an IE-closed subcategory.
\end{itemize}
\end{thma}
Applying \cref{thma:IE} to an affine scheme $X=\Spec R$,
we immediately obtain the result of Enomoto for $\catmod R$
(see \cref{prp:Enomoto's observation}).

It is natural to ask 
when a torsionfree class of $\coh X$ is closed under tensoring with line bundles.
We obtain a complete answer for a connected smooth projective curve.
\begin{thma}[{cf.\ \cref{prp:list torf in coh}}]\label{thma:L-closed in coh C}
Let $C$ be a connected smooth projective curve.
A torsionfree class $\catX$ of $\coh C$ is \emph{not} closed under tensoring with line bundles
if and only if $0 \subsetneq \catX \subsetneq \vect C$,
where $\vect C$ is the category of vector bundles on $C$.
\end{thma}
Combining \cref{thma:L-closed in coh C} with \cref{thma:torf},
we can reduce the classification of torsionfree classes of $\coh C$ to 
the case when it is contained in $\vect C$.
Using this observation,
we completely classify the torsionfree classes of $\coh \bbP^1$ in \cref{prp:torf in coh P^1}.

The second aim of this paper is to classify Serre subcategories of (Quillen's) exact categories
such as the category $\vect X$ of vector bundles over a noetherian scheme $X$,
and the category $\cm R$ of maximal Cohen-Macaulay modules over a commutative noetherian ring $R$.
The main result is as follows and is proved by using \cref{thma:torf,thma:IE}.
\begin{thma}[{$=$ \cref{prp:classify Serre in torf}}]\label{thma:Serre in torf}
Let $X$ be a noetherian scheme having an ample line bundle,
and let $\catX$ be a torsionfree class of $\coh X$ closed under tensoring with line bundles.
Then the assignments $\catS \mapsto \Ass\catS$ and $\Phi \mapsto \coh^{\ass}_{\Phi} X$
give rise to mutually inverse bijections between the following:
\begin{itemize}
\item 
The Serre subcategories of $\catX$
(cf.\ \cref{dfn:Serre in ex cat}).
\item
The specialization-closed subsets of $\Ass \catX$
(with respect to the relative topology induced from $X$).
\end{itemize}
\end{thma}
When $X$ has an ample line bundle and $\catX=\coh X$,
\cref{thma:Serre in torf} recovers Gabriel's classification of Serre subcategories.
\cref{thma:Serre in torf} immediately yields the following classifications of Serre subcategories
of two important exact categories.

\begin{thma}[{$=$ \cref{prp:Serre in cm aff}}]\label{thma:cm}
Let $R$ be a $1$-dimensional noetherian Cohen-Macaulay commutative ring.
Then there is a bijection between the following:
\begin{itemize}
\item 
The Serre subcategories of $\cm R$.
\item
The subsets of the set $\Min R$ of minimal prime ideals.
\qed
\end{itemize}
\end{thma}

A coherent sheaf $\shF$ on a noetherian scheme $X$ of finite Krull dimension
is said to be \emph{pure} 
if for any nonzero quasi-coherent submodule $\shG$,
we have that $\dim \shG = \dim \shF$.
Moreover, it is said to be \emph{maximal} if $\shF=0$ or $\dim \shF =\dim X$.
It is a natural generalization of vector bundles for reducible schemes.
Indeed, vector bundles and maximal pure sheaves coincide on a regular curve.
It is a basic object of 
the studies of moduli spaces and the derived categories for singular curves.
See, for example, \cite{Huy} and \cite{Ishii-Uehara,Kawatani}, respectively.
\begin{thma}[{$=$ \cref{prp:Serre in pure of reduced curve}}]\label{thma:pure on curve}
Let $Z=\bigcup_{i=1}^r C_i$ be a reduced projective curve over a field $\bbk$ 
with the irreducible components $C_i$.
Then there is a bijection between the following:
\begin{itemize}
\item 
The Serre subcategories of the category $\pure X$ of maximal pure sheaves on $X$.
\item
The subsets of $\{C_1,\dots,C_n\}$.
\end{itemize}
\end{thma}
\cref{thma:pure on curve} extends the author's classification of Serre subcategories 
of the category of vector bundles over a connected smooth projective curve in \cite{Saito}
(see \cref{prp:Serre in conn sm proj}).

%
%
%
%
\medskip
\noindent
{\bf Conventions.}
For a category $\catC$,
we denote by $\Hom_{\catC}(M,N)$ the set of morphisms between objects $M$ and $N$ in $\catC$.
In this paper,
we suppose that all subcategories are full subcategories closed under isomorphisms.
Thus, we often identify the subcategories with the subsets of the set of isomorphism classes
of objects in $\catC$.

For a ring $R$ (with identity and associative multiplication),
we denote by $\Mod R$ (resp.\ $\catmod R$) 
the category of (resp.\ finitely generated) (left) $R$-modules.

Let $X$ be a scheme with structure sheaf $\shO_X$.
A point of $X$ is not necessarily assumed to be closed.
For a point $x\in X$,
we denote by $\mm_x$ the maximal ideal of $\shO_{X,x}$ 
and $\kappa(x):=\shO_{X,x}/\mm_x$ the residue field of $x$.
The condition \emph{quasi-compact and quasi-separated}
is abbreviated as \emph{qcqs}.
We denote by $\Qcoh X$ (resp.\ $\coh X$)
the category of quasi-coherent (resp.\ coherent) $\shO_X$-modules.
Let $\shF$ and $\shG$ be quasi-coherent sheaves on $X$.
We denote by $\Hom_{\shO_X}(\shF,\shG)$ the set of $\shO_X$-linear maps from $\shF$ to $\shG$.
The tensor product of $\shF$ and $\shG$ over $\shO_X$ is denoted by $\shF \otimes_{\shO_X} \shG$.
The sheaf of homomorphisms from $\shF$ to $\shG$ is denoted by $\shHom_{\shO_X}(\shF,\shG)$.
If no ambiguity can arise, we will often omit the subscript $\shO_X$.
The support of $\shF$ is the subset of $X$
defined by $\Supp \shF :=\{x\in X \mid \shF_x \ne 0 \}$.
A locally free sheaf on $X$ of finite rank is also called a \emph{vector bundle} over $X$.
In particular, a locally free sheaf of rank $1$ is called a \emph{line bundle}.
Let $\shL$ be a line bundle over $X$.
For a global section $f\in \GG(X,\shL)$ and a point $x\in X$,
we denote by $f(x)$ the image of $f$ in $\shL(x):=\shL_x/\mm_x \shL_x$.
We define the open subset of $X$ by $X_f:=\{x\in X \mid f(x)\ne 0\}$.

Let $X$ be a topological space.
For two points $x,y \in X$,
we say that $x$ is a \emph{specialization} of $y$ or that $y$ is a \emph{generalization} of $x$
if $x$ belongs to the topological closure $\ol{\{y\}}$ of $\{y\}$ in $X$.
We write $x\preceq y$ when $x$ is a specialization of $y$.
The binary relation $\preceq$ defines a preorder on $X$,
which is called the \emph{specialization-order} on $X$.
A subset $A$ of $X$ is \emph{specialization-closed} (resp. \emph{generalization-closed})
if for any $x\in A$ and every its specialization (resp. generalization) $x'\in X$,
we have that $x'\in A$.

\medskip
\noindent
{\bf Acknowledgement.}
First and foremost, the author would like to thank Ryo Takahashi.
The project was initiated in discussions with him, 
where the prototype of \cref{thma:Serre in torf} was obtained.
He is also very grateful to Haruhisa Enomoto for sharing his result,
which is crucial in this paper.
He would like to thank Ryo Kanda and Arashi Sakai for their helpful discussions.
This work is supported by JSPS KAKENHI Grant Number JP21J21767.

\section{Preliminaries}\label{s:preliminaries}

\subsection{Several subcategories of abelian categories}\label{ss:subcat}
In this subsection,
we briefly  recall some properties of subcategories of an abelian category.
\begin{dfn}\label{dfn:several subcat}
Let $\catA$ be an abelian category and $\catX$ its additive subcategory.
\begin{enua}
\item
$\catX$ is said to be \emph{extension-closed} (or \emph{closed under extensions})
if for any exact sequence $0 \to A \to B \to C \to 0$,
we have that $A,C \in \catX$ implies $B\in \catX$.
\item
$\catX$ is said to be \emph{closed under images} (resp.\ \emph{kernels}, resp.\ \emph{cokernels})
if for any morphism $f \colon X \to Y$ in $\catA$ with $X,Y \in \catX$,
we have that $\Ima f \in \catX$ (resp.\ $\Ker f \in \catX$, resp.\ $\Cok f \in \catX$).
\item
$\catX$ is said to be \emph{closed under subobjects} (resp.\ \emph{quotients})
if for any injection $A \inj X$ (resp.\ surjection $X \surj A$) in $\catA$ such that $X\in \catX$,
we have that $A \in \catX$.
\item
$\catX$ is called a \emph{torsionfree class}
if it is closed under subobjects and extensions.
\item
$\catX$ is called a \emph{torsion class}
if it is closed under quotients and extensions.
\item
$\catX$ is called a \emph{Serre subcategory} 
if it is closed under subobjects, quotients and extensions.
\item
$\catX$ is said to be \emph{wide}
if it is closed under kernels, cokernels and extensions.
\item
$\catX$ is said to be \emph{IE-closed}
if it is closed under images and extensions.
\item
$\catX$ is said to be \emph{IKE-closed}
if it is closed under images, kernels and extensions.
\item
$\catX$ is said to be \emph{ICE-closed}
if it is closed under images, cokernels and extensions.
\item
$\catX$ is said to be \emph{narrow}
if it is closed under cokernels and extensions.
\end{enua}
\end{dfn}
It is easy to see that the following implications hold:
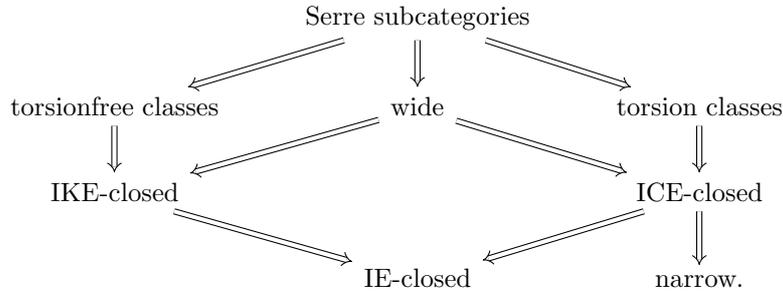
\[
\begin{tikzcd}
&\text{Serre subcategories} \arr{ld,Rightarrow} \arr{d,Rightarrow} \arr{rd,Rightarrow}& \\
\text{torsionfree classes} \arr{d,Rightarrow} & \text{wide} \arr{ld,Rightarrow} \arr{rd,Rightarrow} & \text{torsion classes} \arr{d,Rightarrow}\\
\text{IKE-closed} \arr{rd,Rightarrow} && \text{ICE-closed} \arr{ld,Rightarrow} \arr{d,Rightarrow}\\
& \text{IE-closed} & \text{narrow}.
\end{tikzcd}
\]

In \S \ref{s:subcat in coh}, we will see that
these conditions degenerate into two classes for subcategories of the category of coherent sheaves
under certain assumptions (cf.\ \cref{thm:Stanley-Wang for scheme,thm:IE for scheme}).

We often use the following notation about subcategories of an abelian category.
\begin{dfn}\label{dfn:notation subcat}
Let $\catX$ and $\catY$ be subcategories of an abelian category $\catA$.
\begin{enua}
\item
We denote by $\add \catX$ the subcategory of $\catA$
consisting of direct summands of finite direct sums of objects in $\catX$.
\item
We denote by $\gen{\catX}_{\sub}$ the subcategory of $\catA$
consisting of subobjects of objects of $\catX$.
\item
We denote by $\gen{\catX}_{\quot}$ the subcategory of $\catA$
consisting of quotients of objects of $\catX$.
\item
We denote by $\catX * \catY$ the subcategory of $\catA$
consisting of $M \in \catA$ such that there is an exact sequence
$0\to X \to M \to Y \to 0$ of $\catA$ for some $X\in\catX$ and $Y\in \catY$.
\item
We define the subcategory $\gen{\catX}_{\ext}$ of $\catA$ by
\[
\gen{\catX}_{\ext}:=\bigcup_{n\ge 0} \catX^{n},
\]
where $\catX^{0}:=0$ and $\catX^{n+1}:=\catX^{n} * \catX$.
\item
We define the subcategories of $\catA$ by
$\F(\catX):=\gen{\gen{\catX}_{\sub}}_{\ext}$ and $\T(\catX):=\gen{\gen{\catX}_{\quot}}_{\ext}$.
\end{enua}
\end{dfn}

\begin{prp}
The following hold for a subcategory $\catX$ of an abelian category $\catA$.
\begin{enua}
\item
$\F(\catX)$ is the smallest torsionfree class of $\catA$ containing $\catX$.
\item
$\T(\catX)$ is the smallest torsion class of $\catA$ containing $\catX$.
\end{enua}
\end{prp}
\begin{proof}
See \cite[Lemma 2.5 and 2.6]{IE=torf}.
\end{proof}

The following observation is often used in this paper.
\begin{lem}\label{lem:half ex}
Let $\catA$ be an abelian category and $\catX$ its subcategory.
\begin{enua}
\item
Suppose that $\catX$ is a torsionfree class.
Then for any exact sequence $0\to L \to M\to N$ in $\catA$,
we have that $L,N \in\catX$ implies $M\in\catX$.
\item
Suppose that $\catX$ is a torsion class.
Then for any exact sequence $L \to M\to N\to 0$ in $\catA$,
we have that $L,N \in\catX$ implies $M\in\catX$.
\item
Suppose that $\catX$ is a Serre subcategory.
Then for any exact sequence $L \to M\to N$ in $\catA$,
we have that $L,N \in\catX$ implies $M\in\catX$.
\end{enua}
\end{lem}
\begin{proof}
We omit the proof since it is straightforward.
\end{proof}

Finally, we introduce the notion of Serre subcategories of an extension-closed category,
which is one of the main topics of this paper.
\begin{dfn}\label{dfn:Serre in ex cat}
Let $\catX$ be an extension-closed subcategory of an abelian category $\catA$.
\begin{enua}
\item
A \emph{conflation} of $\catX$ is 
an exact sequence $0 \to A \to B \to C \to 0$ of $\catA$
such that $A$, $B$ and $C$ belong to $\catX$.
\item
An additive subcategory $\catS$ of $\catX$ is called a \emph{Serre subcategory of $\catX$}
if for any conflation $0 \to X \to Y \to Z \to 0$ of $\catX$,
we have that $Y\in \catS$ if and only if both $X$ and $Z$ belong to $\catS$.
\end{enua}
\end{dfn}

\subsection{Ample families}\label{ss:ample family}
In this subsection,
we recall the notion of ample families on a quasi-compact and quasi-separated (qcqs) scheme,
which generalizes the notion of ample line bundles.
We often use the following notation about line bundles.
\begin{nt}
Let $\{\shL_{\alpha}\}_{\alpha\in \LL}$ be 
a nonempty family of line bundles on a scheme $X$,
and let $\shF$ be a quasi-coherent $\shO_X$-module.
We use the following notation:
\[
\shF(\alpha_1,\dots,\alpha_r; n_1,\dots,n_r)
:=\shF\otimes \shL_{\alpha_1}^{\otimes n_1}\otimes\cdots\otimes \shL_{\alpha_r}^{\otimes n_r},\quad
\shO(\alpha_1,\dots,\alpha_r; n_1,\dots,n_r)
:=\shL_{\alpha_1}^{\otimes n_1}\otimes\cdots\otimes \shL_{\alpha_r}^{\otimes n_r},
\]
where 
$r$ is a positive integer, $n_i$ is an integer,
and $\{\alpha_1,\dots,\alpha_r\}$ is a finite subset of $\LL$.

The following formulas are often used in this paper:
\begin{itemize}
\item 
$\shF(\alpha_1,\dots,\alpha_r; n_1,\dots,n_r) \otimes \shO(\alpha_1,\dots,\alpha_r; m_1,\dots,m_r)\iso \shF(\alpha_1,\dots,\alpha_r; n_1+m_1,\dots,n_r+m_r)$.
\item
$\shHom(\shO(\alpha_1,\dots,\alpha_r; -n_1,\dots,-n_r),\shF) \iso \shF(\alpha_1,\dots,\alpha_r; n_1,\dots,n_r)$.
\end{itemize}
\end{nt}

An affine scheme $X$ possesses the following two significant features.
\begin{itemize}
\item 
The set $\{X_f \mid f \in \GG(X,\shO_X)\}$ is an open basis of $X$.
\item
Any coherent $\shO_X$-module $\shF$ is a quotient of a free $\shO_X$-module,
that is, there is a surjection $\shO_X^{\oplus n} \surj \shF$.
\end{itemize}
The notion of divisorial schemes generalizes these properties
by using a family of line bundles.

\begin{dfn}
Let $X$ be a qcqs scheme.
\begin{enua}
\item
A non-empty family $\calL=\{\shL_{\alpha}\}_{\alpha \in \LL}$ of line bundles on $X$
is called an \emph{ample family} if the following set is an open basis of $X$:
\[
\calB_{\calL}:=\{X_f \mid f \in \GG(X,\shO(\alpha;n)), n>0, \alpha \in \LL\}.
\]
\item
$X$ is said to be \emph{divisorial} if it admits an ample family.
\end{enua}
\end{dfn}

\begin{fct}[{cf.\ \cite[Chapter II, Proposition 2.2.3]{SGA6}}]\label{fct:char ample family}
Let $X$ be a qcqs scheme.
The following are equivalent for
a non-empty family $\calL=\{\shL_{\alpha}\}_{\alpha \in \LL}$ of line bundles on $X$. 
\begin{enur}
\item
$\calL$ is an ample family.
\item
The set $\{X_f \in \calB_{\calL} \mid \text{$X_f$ is affine}\}$ is an open basis of $X$.
\item
There is an affine open cover $\{U_i\}_{i\in I}$ of $X$ such that $U_i=X_{f_i}$ 
for some $\alpha_i \in \LL$, $n_i>0$ and $f_i\in \GG(X,\shO(\alpha_i;n_i))$.
\item
For any quasi-coherent $\shO_X$-module $\shF$ of finite type,
there exist integers $n_{\alpha}>0$, $k_{\alpha} \ge 0$ and a surjective $\shO_X$-linear map
$\bigoplus_{\alpha\in\LL} \shO(\alpha;-n_{\alpha})^{\oplus k_{\alpha}} \surj \shF$.
\end{enur}
\end{fct}

\begin{rmk}\label{rmk:fin ample family}
Let $X$ be a qcqs scheme
and $\{\shL_{\alpha}\}_{\alpha \in \LL}$ an ample family.
\begin{enua}
\item
There is a finite subset $\{\alpha_1,\dots, \alpha_r\} \subseteq \LL$
such that $\{\shL_{\alpha_i}\}_{i=1}^r$ is also an ample family.
It follows from \cref{fct:char ample family} (iii) and the quasi-compactness of $X$.
\item
For any quasi-coherent $\shO_X$-module $\shF$ of finite type,
there exist a finite subset $\{\alpha_1,\dots, \alpha_r\} \subseteq \LL$, 
integers $n_{i}>0$, $k_{i} \ge0$ and a surjective $\shO_X$-linear map
$\bigoplus_{i=1}^r \shO(\alpha_i;-n_i)^{\oplus k_i} \surj \shF$.
It follows from (1) and \cref{fct:char ample family} (iv).
\end{enua}
\end{rmk}

In particular, the following corollary holds.
\begin{cor}\label{cor:divisorial resol}
Let $X$ be a divisorial scheme.
Then for any quasi-coherent $\shO_X$-module $\shF$ of finite type,
there exist finitely many line bundles $\shL_1, \dots, \shL_r$
and a surjective $\shO_X$-linear map $\oplus_{i=1}^r \shL_i \surj \shF$.
\end{cor}

\begin{ex}
We give examples of ample families and divisorial schemes.
\begin{enua}
\item
For a line bundle $\shL$ on a qcqs scheme,
it is ample if and only if the singleton $\{\shL\}$ is an ample family
(cf.\ \cite[Proposition 13.47]{GW}).
\item
For a qcqs scheme $X$,
it is quasi-affine (that is, there exists an open embedding of $X$ into an affine scheme)
if and only if the structure sheaf $\shO_X$ is ample
(cf.\ \cite[Proposition and Definition 13.78]{GW}).
\item
A quasi-projective scheme $X$ over a commutative ring $R$
(that is, there exists an $R$-immersion $X \inj \bbP^n_R$ 
into the projective space $\bbP^n_R$ over $R$ for some $n\ge 0$)
has an ample line bundle
(cf.\ \cite[Summary 13.71]{GW}).
\item
A noetherian separated regular scheme is divisorial
(cf.\ \cite[Chapter II, Corollaire 2.2.7.1]{SGA6}).
\end{enua}
\end{ex}

\section{Classifying several subcategories of the category of coherent sheaves}\label{s:subcat in coh}
In this section,
we classify several subcategories of the category $\coh X$ of coherent sheaves
on a divisorial noetherian scheme $X$.
We first introduce the notion of subcategories closed under tensoring with line bundles
in \S \ref{ss:L-closed subcat}.
Next, we develop the theory up to tensoring with line bundles
in \S \ref{ss:torf in coh}--\ref{ss:IE in coh}.
This allows us to extend various techniques 
for the module category on a commutative noetherian ring
to the category $\coh X$ of coherent sheaves 
and classify several classes of subcategories closed under tensoring with line bundles of $\coh X$.
See \S \ref{ss:quasi-affine} for the case of commutative noetherian rings.

In the next section, we will use the theory established in this section 
to classify the Serre subcategories of a torsionfree class in $\coh X$.

\subsection{Subcategories closed under tensoring with line bundles}\label{ss:L-closed subcat}
In this subsection, 
we introduce the notion of subcategories closed under tensoring with line bundles
and give criteria for subcategories to be closed under tensoring with line bundles.
This notion plays a central role in this paper.

\begin{dfn}\label{dfn:L-closed}
Let $X$ be a locally noetherian scheme. 
A subcategory $\catX$ of $\coh X$ is said to be \emph{closed under tensoring with line bundles}
if $\shF \otimes \shL \in \catX$ for any $\shF\in\catX$ and any line bundle $\shL$ on $X$.
\end{dfn}

If $X$ has an ample family, 
there is a convenient way to determine 
whether torsion(free) classes are closed under tensoring with line bundles.
\begin{lem}\label{lem:char torf closed under tensoring}
Let $X$ be a noetherian scheme 
having an ample family $\{\shL_{\alpha}\}_{\alpha \in \LL}$.
The following are equivalent for an additive subcategory $\catX$ of $\coh X$
closed under kernels.
\begin{enur}
\item
$\catX$ is closed under tensoring with line bundles.
\item
$\shF\otimes \shL_{\alpha} \in \catX$ holds for any $\shF\in\catX$ and any $\alpha \in \LL$.
\item
$\catX$ is $\shHom(\coh X,-)$-closed, that is,
$\shHom(\shG,\shF) \in \catX$ holds for any $\shF \in \catX$ and $\shG \in\coh X$.
\end{enur}
\end{lem}
\begin{proof}
It is clear that (i)$\imply$(ii).
The implication (iii)$\imply$(i) follows from 
the isomorphism $\shF \otimes \shL \iso \shHom\left(\shL^{\vee},\shF\right)$,
where $\shF \in \coh X$ and $\shL$ is a line bundle.
We prove (ii)$\imply$(iii).
Let $\shF \in \catX$ and $\shG \in \coh X$.
The following exact sequence exists in $\coh X$ by \cref{rmk:fin ample family}:
\[
\bigoplus_{j=1}^s \shO(\beta_j;-m_j)^{\oplus l_j} \to \bigoplus_{i=1}^t \shO(\alpha_i;-n_i)^{\oplus k_i} \to \shG \to 0,
\]
where $\alpha_i, \beta_j \in \LL$, $n_i,m_j >0$ and $k_i,l_j\ge 0$.
Applying the functor $\shHom(-,\shF)$ to this sequence,
we have the following exact sequence in $\coh X$:
\[
0 \to \shHom(\shG,\shF) \to \bigoplus_{i=1}^t \shF(\alpha_i;n_i)^{\oplus k_i} \to \bigoplus_{j=1}^s \shF(\beta_j;m_j)^{\oplus l_j}.
\]
Then we have $\shHom(\shG,\shF) \in \catX$ 
since $\catX$ is an additive subcategory closed under taking kernels and the condition (ii) holds.
\end{proof}

\begin{lem}\label{lem:char tors closed under tensoring}
Let $X$ be a noetherian scheme 
having an ample family $\{\shL_{\alpha}\}_{\alpha \in \LL}$.
The following are equivalent for an additive subcategory $\catX$ of $\coh X$
closed under cokernels.
\begin{enur}
\item
$\catX$ is closed under tensoring with line bundles.
\item
$\shF\otimes \shL_{\alpha}^{\vee} \in \catX$ holds for any $\shF\in\catX$ and any $\alpha \in \LL$.
\item
$\catX$ is tensor-ideal, that is,
$\shF\otimes \shG \in \catX$ holds for any $\shF \in \catX$ and $\shG \in \coh X$.
\end{enur}
\end{lem}
\begin{proof}
It is clear that (iii)$\imply$(i)$\imply$(ii).
We prove (ii)$\imply$(iii).
Let $\shF \in \catX$ and $\shG\in \coh X$.
The following exact sequence exists in $\coh X$ by \cref{rmk:fin ample family}:
\[
\bigoplus_{j=1}^s \shO(\beta_j;-m_j)^{\oplus l_j} \to \bigoplus_{i=1}^t \shO(\alpha_i;-n_i)^{\oplus k_i} \to \shG \to 0,
\]
where $\alpha_i, \beta_j \in \LL$, $n_i,m_j >0$ and $k_i,l_j\ge 0$.
Tensoring this sequence with $\shF$,
we have the following exact sequence in $\coh X$:
\[
\bigoplus_{j=1}^s \shF(\beta_j;-m_j)^{\oplus l_j} \to \bigoplus_{i=1}^t \shF(\alpha_i;-n_i)^{\oplus k_i} \to \shF \otimes \shG \to 0.
\]
Then we have $\shF \otimes \shG \in \catX$ 
since $\catX$ is an additive subcategory closed under taking cokernels 
and the condition (ii) holds.
\end{proof}

We will give a similar way to determine 
whether IE-closed categories are closed under tensoring with line bundles
in \cref{lem:char IE closed under tensoring}.

The following categories are frequently used in this paper.
\begin{ex}
Let $X$ be a locally noetherian scheme.
Let $\mathbf{P}$ be a property of modules over a commutative local ring.
Then the subcategory $\catX(\mathbf{P})$ of $\coh X$ defined by the following 
is closed under tensoring with line bundles:
\[
\catX(\mathbf{P})
:=\{\shF \in \coh X \mid \text{$\shF_x$ satisfies $\mathbf{P}$ for any $x\in X$} \}.
\]
In particular,
the following subcategories of $\coh X$ are closed under tensoring with line bundles.
\begin{itemize}
\item
The category $\vect X$ of vector bundles on $X$.
\item
The category $\tf X$ of locally torsionfree sheaves on $X$ (cf.\ \S \ref{ss:torf sh}).
\item
The category $\cm X$ of maximal Cohen-Macaulay sheaves on $X$ (cf.\ \S \ref{ss:torf pure CM}).
\end{itemize}
\end{ex}

\begin{ex}
Let $X$ be a locally noetherian scheme.
The following subcategories of $\coh X$ are closed under tensoring with line bundles.
\begin{itemize}
\item
$\coh_Z X:=\{\shF \in \coh X \mid \Supp \shF \subseteq Z\}$ for a subset $Z \subseteq X$.
\item
$\coh_{\Phi}^{\ass} X:=\{\shF \in \coh X \mid \Ass \shF \subseteq \Phi\}$ 
for a subset $\Phi \subseteq X$.
\end{itemize}
\end{ex}

\subsection{Classifying torsionfree classes closed under tensoring with line bundles}\label{ss:torf in coh}
Let $X$ be a divisorial noetherian scheme.
In this subsection,
we classify torsionfree classes of $\coh X$ closed under tensoring with line bundles.
More precisely,
consider the following assignments:
\begin{itemize}
\item
For a subset $\Phi$ of $X$,
define a subcategory $\coh^{\ass}_{\Phi} X$ of $\coh X$ by
\[
\coh^{\ass}_{\Phi} X :=\{\shF \in \coh X \mid \Ass \shF \subseteq \Phi\}.
\]
\item
For a subcategory $\catX$ of $\coh X$,
define a subset $\Ass \catX$ of $X$ by
\[
\Ass \catX := \bigcup_{\shF \in \catX} \Ass \shF.
\]
\end{itemize}
We will prove that these assignments yield mutually inverse bijections
between the power set of $X$ 
and the set of torsionfree classes of $\coh X$ closed under tensoring with line bundles
(cf.\ \cref{thm:Takahashi for scheme}).
This generalizes Takahashi's classification of torsionfree classes of 
the category of finitely generated modules over a commutative noetherian ring
(cf.\ \cref{fct:Takahashi torf}).
We refer to \S \ref{ss:Ass} 
for the notion of associated points and the notation used in this subsection.

First of all, we extend various techniques for the module category on a commutative noetherian ring
to the category $\coh X$ of coherent sheaves, up to tensoring with line bundles.
\begin{lem}\label{lem:germ loc sec affine}
Let $X$ be a scheme and $\shF$ a quasi-coherent $\shO_X$-module.
Let $s\in \shF_x$ be a germ at some point $x\in X$.
Then for any affine neighborhood $U$ of $x$,
there are sections $s'\in \shF(U)$ and $a\in\shO_X(U)$
such that $s'_x=a_xs$ and $a_x \in \shO_{X,x}$ is invertible.
\end{lem}
\begin{proof}
It follows from $\shF_x=\shF(U)_{\pp_x}$,
where $\pp_x$ is the prime ideal of $\shO_X(U)$ corresponding to $x \in U$.
\end{proof}

The following lemma asserts that
a local homomorphism $\shF_x \to \shG_x$ between the stalks of coherent sheaves
can be extended to a global homomorphism $\shF \to \shG$, \emph{up to tensoring with line bundles}.
This lemma plays an important role throughout this section.
\begin{lem}\label{lem:lift local hom to global hom}
Let $X$ be a qcqs scheme having an ample family $\{\shL_{\alpha}\}_{\alpha\in \LL}$,
$\shF$ an $\shO_X$-module of finite presentation,
and $\shG$ a quasi-coherent $\shO_X$-module.
For any point $x\in X$ and any $\shO_{X,x}$-linear map $\phi\colon \shF_x \to \shG_x$,
there exist an $\shO_X$-linear map $\psi\colon \shF\to \shG(\alpha;n)$ 
for some $\alpha \in \LL$ and $n>0$,
and an $\shO_{X,x}$-isomorphism $\mu \colon \shG(\alpha;n)_x \isoto \shG_x$
such that $\phi=\mu\circ\psi_x$.
\end{lem}
\begin{proof}
Since $\{\shL_{\alpha}\}_{\alpha\in \LL}$ is an ample family,
there are $\alpha \in \LL$, $m>0$ and $f\in \GG\left(X,\shL_{\alpha}^{\otimes m}\right)$ 
such that $X_f$ is an affine open neighborhood of $x$.
Since $\shF$ is of finite presentation and $\shG$ is quasi-coherent,
the $\shO_X$-module $\shHom(\shF,\shG)$ is quasi-coherent and we have an $\shO_{X,x}$-isomorphism
$\shHom_{\shO_X}(\shF,\shG)_x \iso \Hom_{\shO_{X,x}}(\shF_x,\shG_x)$
by \cite[Proposition 7.27 and 7.29]{GW}.
By \cref{lem:germ loc sec affine},
there are sections $a \in \shO_X(X_f)$ and 
$\psi' \in \GG\left(X_f,\shHom_{\shO_X}(\shF,\shG)\right)$ such that
$\psi'_x = a_x \cdot \phi$ and $a_x \in \shO_{X,x}$ is invertible.
Then we have a global section 
$\psi\in \GG\left(X,\shHom_{\shO_X}(\shF,\shG)\otimes\left(\shL_{\alpha}^{\otimes m}\right)^{\otimes n} \right)$
such that $\psi_{|X_f}= \psi' \otimes f^{\otimes n}$ for some $n>0$
by \cite[Theorem 7.22]{GW}.
We also denote by $\psi\colon \shF \to \shG(\alpha;mn)$
the $\shO_X$-linear map corresponding to $\psi$ by the following isomorphism
(cf.\ \cite[Proposition 7.7]{GW}):
\[
\GG\left(X,\shHom_{\shO_X}(\shF,\shG)\otimes\left(\shL_{\alpha}^{\otimes m}\right)^{\otimes n}\right) 
\isoto \GG\left(X, \shHom_{\shO_X}(\shF,\shG(\alpha;mn))\right)
= \Hom_{\shO_X}(\shF,\shG(\alpha;mn)).
\]
There is the following $\shO_{X|X_f}$-isomorphism by \cite[Remark 13.46 (2)]{GW}:
\[
u \colon \shO_{X|X_f} \isoto (\shL_{\alpha}^{\otimes m})_{|X_f},\quad 1 \mapsto f_{|X_f}.
\]
Then we have the following commutative diagram:
\[
\begin{tikzcd}[column sep=30pt]
\shF_x \arr{d,"\phi"'} \arr{rrr,"\psi_x"} \arr{rd,"{\psi'_x}"}& & & \shG(\alpha;mn)_x \arr{d,"\iso"}\\
\shG_x \arr{r,"a_x{\cdot}-"',"\iso"}& \shG_x \arr{r,"\iso"} & \shG_x\otimes \shO_{X,x}^{\otimes n} \arr{r,"\id\otimes u_{x}^{\otimes n}"',"\iso"} & \shG_x \otimes \left(\shL_{\alpha}^{\otimes m}\right)_x^{\otimes n}.
\end{tikzcd}
\]
This finishes the proof.
\end{proof}

Let $X$ be a scheme.
For any point $x\in X$, we denote by $Z_x$
the reduced closed subscheme of $X$ whose underlying topological space is $\ol{\{x\}}$
(cf.\ \cite[Proposition 3.52]{GW}).

In the affine case,
for any associated prime ideal $\pp$ of a module $M$ over a commutative ring $R$,
there is an injective $R$-linear map $R/\pp \inj M$ by definition.
On the other hand,
an associated point $x$ of a quasi-coherent sheaf $\shF$ on a scheme $X$
does not yield an injective $\shO_X$-linear map $\shO_{Z_x} \inj \shF$
(see \cite[Example B.11]{CS}).
However,
we can obtain such an embedding, up to tensoring with line bundles, 
as the following proposition shows.
See \cite[Proposition B.12]{CS} for another kind of generalization of this.
\begin{prp}\label{prp:ass pt inj}
Let $X$ be a noetherian scheme having an ample family $\{\shL_{\alpha}\}_{\alpha\in \LL}$,
and let $\shF$ be a quasi-coherent sheaf on $X$.
For any $x\in \Ass \shF$,
there is an injective $\shO_X$-linear map $\shO_{Z_x} \inj \shF(\alpha; n)$ 
for some $\alpha \in \LL$ and $n>0$.
\end{prp}
\begin{proof}
For any associated point $x$ of $\shF$,
there is an injective $\shO_{X,x}$-linear map $\kk(x)\inj \shF_x$.
By \cref{lem:lift local hom to global hom},
we have an $\shO_X$-linear map $\psi \colon \shO_{Z_x} \to \shF(\alpha;n)$
for some $\alpha\in \LL$ and $n>0$
such that $\psi_x$ is injective.
Then $\psi$ is injective since $\Ass_X \shO_{Z_x}=\{x\}$.
\end{proof}


We now come back to classifying torsionfree classes closed under tensoring with line bundles.
We prove the lemmas needed to prove \cref{thm:Takahashi for scheme} below.
\begin{lem}\label{lem:char x in Ass}
Let $X$ be a divisorial noetherian scheme and
a torsionfree class $\catX$ of $\coh X$ closed under tensoring with line bundles.
For any point $x\in X$,
we have that $x\in \Ass \catX$ if and only if $\shO_{Z_x} \in \catX$.
\end{lem}
\begin{proof}
If $\shO_{Z_x} \in \catX$,
then we have $x \in \Ass_X \shO_{Z_x} \subseteq \Ass \catX$.
Conversely, suppose that $x\in \Ass \catX$. 
Then there exists $\shF \in \catX$ such that $x\in \Ass \shF$.
By \cref{prp:ass pt inj},
there is an injective $\shO_X$-linear map $\shO_{Z_x} \inj \shF\otimes \shL^{\otimes n}$ 
for some line bundle $\shL$ and $n>0$.
Since $\catX$ is a torsionfree class closed under tensoring with line bundles,
we conclude that $\shO_{Z_x} \in \catX$.
\end{proof}

\begin{lem}\label{lem:torf Ass=1}
Let $X$ be a divisorial noetherian scheme.
For any torsionfree class $\catX$ of $\coh X$ closed under tensoring with line bundles,
the following holds:
\[
\catX=\gen{\shF \in \catX \mid \# \Ass \shF =1}_{\ext}.
\]
\end{lem}
\begin{proof}
The proof is inspired by \cite[Proposition 3.8]{Iyama-Kimura}.
Let $\shF\in \catX$.
Since $\shF$ is a noetherian object of $\Qcoh X$,
it is enough to show that, if $\shF\ne 0$, 
then there is a non-zero quasi-coherent $\shO_X$-submodule $\shG$ of $\shF$
such that $\# \Ass \shG =1$ and $\shF/\shG\in \catX$.
Let $x$ be a minimal element of $\Ass \shF$ with respect to the specialization-order.
That is, if $y\in \ol{\{x\}}$ such that $y\in \Ass \shF$, then $x=y$ holds.
Such an element $x$ certainly exists since $\Ass \shF$ is a nonempty finite set.
Since $X$ is divisorial and the ideal sheaf $\shI_{Z_x}$ of $Z_x$ is of finite type,
there is a surjective $\shO_X$-linear map
$\oplus_{i=1}^r \shL_i \surj \shI_{Z_x}$
for some line bundles $\shL_i$ by \cref{cor:divisorial resol}.
Thus, we have the following exact sequence in $\coh X$:
\[
\bigoplus_{i=1}^r \shL_i \to \shO_X \to \shO_{Z_x} \to 0.
\]
Applying the functor $\shHom_{\shO_X}(-,\shF)$ to this sequence,
we obtain the following exact sequence in $\coh X$:
\[
0 \to \shHom_{\shO_X}(\shO_{Z_x},\shF) \to \shF \to \bigoplus_{i=1}^r \shF \otimes \shL_i^{\vee}.
\]
Then $\shG :=\shHom_{\shO_X}(\shO_{Z_x},\shF)$ is a nonzero coherent $\shO_X$-module
since there is an injection $\kk(x) \inj \shF_x$
and $\shHom_{\shO_X}(\shO_{Z_x},\shF)_x \iso \Hom_{\shO_{X,x}}(\kk(x),\shF_x)$.
Because $\catX$ is a torsionfree class closed under tensoring with line bundles,
we have that $\shG, \shF/\shG \in \catX$.
It remains to show that $\# \Ass \shG =1$.
Since $\Supp \shG \subseteq \ol{\{x\}}$,
we have that $\Ass \shG \subseteq \Ass \shF \cap \ol{\{x\}} = \{x\}$.
Here, the last equality follows from the minimality of $x$.
Thus, we obtain that $\# \Ass\shG =1$ as $\shG$ is nonzero.
\end{proof}

\begin{lem}\label{lem:Ass=1 Z_x}
Let $X$ be a divisorial noetherian scheme and
a torsionfree class $\catX$ of $\coh X$ closed under tensoring with line bundles.
Suppose that $\shO_{Z_x} \in \catX$ for some point $x\in X$.
Then any coherent $\shO_X$-module $\shF$ such that $\Ass \shF=\{x\}$ belongs to $\catX$.
\end{lem}
\begin{proof}
The proof is inspired by \cite[Lemma 4.2]{Takahashi}.
Let $\shF$ be a coherent $\shO_X$-module such that $\Ass \shF=\{x\}$.
Let $\phi_{0,1},\dots, \phi_{0,k_0}$ be 
a system of generators of $\Hom_{\shO_{X,x}}(\shF_x,\kk(x))$
as an $\shO_{X,x}$-module.
Put $\phi_0:=\begin{bsmatrix}\phi_{0,1}& \cdots & \phi_{0.k_0}\end{bsmatrix}^t \colon \shF_x \to \kk(x)^{\oplus k_0}$.
Then there are line bundles $\shL$
and $\shO_X$-linear maps $\psi_0\colon \shF \to \shO_{Z_x}^{\oplus k_0} \otimes \shL$
such that $\phi_{0} \colon \shF_x \xr{\psi_{0,x}} \left(\shO_{Z_x}^{\oplus k_0}\otimes \shL\right)_x \iso \kk(x)^{\oplus k_0}$ by \cref{lem:lift local hom to global hom}.
Setting $\shG_0:=\shO_{Z_x}^{\oplus k_0} \otimes \shL \in \catX$ and $\shF_1:=\Ker(\psi_0)$,
we have the following exact sequence:
\[
0\to \shF_1 \to \shF \to \shG_0.
\]
Iterating this procedure, we obtain exact sequences
\begin{equation}\label{eq:Ass=1 Z_x 1}
0 \to \shF_{i+1} \xr{\iota_i} \shF_i  \xr{\psi_i} \shG_i,
\end{equation}
which satisfies the following:
\begin{itemize}
\item
$\shF_0=\shF$ and $\shG_i\in\catX$ for any $i$.
\item
There is an $\shO_{X,x}$-isomorphism $\shG_{i,x} \iso \kk(x)^{\oplus k_i}$
such that the following diagram commutes:
\[
\begin{tikzcd}[ampersand replacement=\&]
\shF_{i,x} \arr{r,"\psi_{i,x}"} \arr{rd,"{\begin{bsmatrix}\phi_{i,1} \\ \vdots \\ \phi_{i,k_i} \end{bsmatrix}}"'} \& \shG_{i,x} \arr{d,"\iso"}\\
\& \kk(x)^{\oplus k_i}
\end{tikzcd}
\]
\item
$\Hom_{\shO_{X,x}}(\shF_{i,x},\kk(x))$ is generated by $\phi_{i,1},\dots \phi_{i,k_i}$
as an $\shO_{X,x}$-module.
\end{itemize}
Now, we have the descending chain
$\shF=\shF_0 \supseteq \shF_1 \supseteq \shF_2 \supseteq \cdots$
of coherent $\shO_X$-submodules of $\shF$.
Then $\shF_x$ is an $\shO_{X,x}$-module of finite length since $\Ass_{\shO_{X,x}} \shF_x=\{\mm_x\}$.
Thus, the descending chain
$\shF_x=\shF_{0,x} \supseteq \shF_{1,x} \supseteq \shF_{2,x} \supseteq \cdots$
becomes stationary,
that is, there is some $n \ge 0$ such that $\shF_{n,x}=\shF_{n+1,x}=\shF_{n+2,x}=\cdots$.
The exact sequence
\[
0 \to \shF_{n+1,x} \xr{\iso} \shF_{n,x} \xr{\begin{bsmatrix} \phi_{n,1} \\ \vdots \\ \phi_{n,k_{n}} \end{bsmatrix}} \kk(x)^{\oplus k_{n}}
\]
shows that $\phi_{n,1}=\cdots=\phi_{n,k_{n}}=0$,
and hence we have $\Hom_{\shO_{X,x}}(\shF_{n,x},\kk(x))=0$.
This implies $\shF_{n,x}=0$. 
Thus, we obtain $\shF_{n}=0$ because $\Ass \shF_{n} \subseteq \Ass \shF=\{x\}$.
Then we conclude that $\shF \in \catX$ 
by the exact sequences \eqref{eq:Ass=1 Z_x 1} and \cref{lem:half ex}.
\end{proof}

\begin{lem}\label{lem:Ass catX inclusion}
Let $X$ be a divisorial noetherian scheme.
For any torsionfree classes $\catX, \catY \subseteq \coh X$ 
closed under tensoring with line bundles,
we have that $\catX \subseteq \catY$ if and only if $\Ass\catX \subseteq \Ass\catY$.
In particular, we have that $\catX = \catY$ if and only if $\Ass\catX = \Ass\catY$.
\end{lem}
\begin{proof}
The only if part is obvious.
Suppose that $\Ass\catX \subseteq \Ass\catY$.
It is enough to show that any $\shF\in\catX$ such that $\Ass\shF=\{x\}$ belongs to $\catY$
by \cref{lem:torf Ass=1}.
Since $x \in \Ass\shF \subseteq \Ass\catX \subseteq \Ass \catY$,
we have that $\shO_{Z_x} \in \catY$ by \cref{lem:char x in Ass}.
Thus, we obtain $\shF \in \catY$ by \cref{lem:Ass=1 Z_x}.
\end{proof}

We now classify the torsionfree classes closed under tensoring with line bundles.
\begin{thm}\label{thm:Takahashi for scheme}
Let $X$ be a divisorial noetherian scheme.
\begin{enua}
\item
$\coh^{\ass}_{\Phi} X$ is a torsionfree class of $\coh X$ closed under tensoring with line bundles
for any subset $\Phi$ of $X$.
\item
$\Phi=\Ass\left(\coh^{\ass}_{\Phi} X\right)$ holds
for any subset $\Phi$ of $X$.
\item
$\catX=\coh^{\ass}_{\Ass\catX} X$ holds
for any torsionfree class $\catX$ of $\coh X$ closed under tensoring with line bundles.
\item
The assignments $\Phi \mapsto \coh^{\ass}_{\Phi} X$ and $\catX \mapsto \Ass\catX$
give mutually inverse bijections between the following sets:
\begin{itemize}
\item
The power set of $X$.
\item 
The set of torsionfree classes of $\coh X$ closed under tensoring with line bundles.
\end{itemize}
\end{enua}
\end{thm}
\begin{proof}
(1)
It is obvious.

(2)
For any $x\in \Phi$, we have that $\shO_{Z_x} \in \coh^{\ass}_{\Phi} X$.
Thus $\Phi \subseteq \Ass\left(\coh^{\ass}_{\Phi} X\right)$ holds.
If $x\in \Ass\left(\coh^{\ass}_{\Phi} X\right)$,
then $\shO_{Z_x} \in \coh^{\ass}_{\Phi} X$ by \cref{lem:char x in Ass}.
Then we have $x \in \Ass_X \shO_{Z_x} \subseteq \Phi$, 
and hence $\Ass\left(\coh^{\ass}_{\Phi} X\right) \subseteq \Phi$ holds.
Therefore, we obtain $\Phi = \Ass\left(\coh^{\ass}_{\Phi} X\right)$.

(3)
It follows from
$\Ass \catX = \Ass(\coh^{\ass}_{\Ass\catX} X)$ and \cref{lem:Ass catX inclusion}.

(4)
It follows from (1)--(3).
\end{proof}

\begin{rmk}
In \cite{CS},
torsionfree classes of $\coh X$ closed under tensoring with line bundles were also studied
and some kinds of them were classified via cotilting objects.
It seems that their study is strongly related to the results in this subsection,
but it does not overlap.
\end{rmk}
\subsection{Classifying torsion classes closed under tensoring with line bundles}\label{ss:tor in coh}
Let $X$ be a divisorial noetherian scheme.
In this subsection,
we classify torsion classes of $\coh X$ closed under tensoring with line bundles.
More precisely, we prove that torsion classes closed under tensoring with line bundles
coincide with Serre subcategories (cf.\ \cref{thm:Stanley-Wang for scheme}).
This generalizes Stanley and Wang's classification of torsion classes of 
the category of finitely generated modules over a commutative noetherian ring 
(cf.\ \cref{fct:tor=Serre}).

%

We first recall 
Gabriel's classification of Serre subcategories of the category of coherent sheaves.
Let $X$ be a noetherian scheme.
Consider the following assignments:
\begin{itemize}
\item
For a subset $\Phi$ of $X$,
define a subcategory $\coh_{\Phi} X$ of $\coh X$ by
\[
\coh_{\Phi} X :=\{\shF \in \coh X \mid \Supp \shF \subseteq \Phi\}.
\]
\item
For a subcategory $\catX$ of $\coh X$,
define a subset $\Supp \catX$ of $X$ by
\[
\Supp \catX := \bigcup_{\shF \in \catX} \Supp \shF.
\]
\end{itemize}

\begin{fct}[{\cite[Proposition VI.2.4]{Gabriel}, see also \cite[Theorem 5.5]{atom}}]\label{fct:Gabriel}
Let $X$ be a noetherian scheme.
The assignments $\catX \mapsto \Supp\catX$ and $Z \mapsto \coh_{Z} X$
give mutually inverse bijections between the following sets:
\begin{itemize}
\item
The set of Serre subcategories of $\coh X$.
\item
The set of specialization-closed subsets of $X$.
\end{itemize}
\end{fct}

\begin{rmk}\label{rmk:Serre L-closed}
Let $X$ be a noetherian scheme.
Then any Serre subcategory of $\coh X$ is of the form $\coh_Z X$ by \cref{fct:Gabriel}.
In particular, every Serre subcategory of $\coh X$ is closed under tensoring with line bundles.
If $X$ has an ample line bundle, we can prove this without \cref{fct:Gabriel}
(see \cref{lem:Serre in torf ample stable}).
However, the author does not know a direct proof for the general case
without \cref{fct:Gabriel} at the moment.
\end{rmk}

\begin{rmk}
Let $X$ be a noetherian scheme.
For a specialization-closed subset $Z$ of $X$,
we have that $\coh_Z^{\ass} X = \coh_Z X$.
Indeed,
it is enough to show that 
$\Ass \shF \subseteq Z$ if and only if $\Supp \shF \subseteq Z$ 
for any $\shF \in \coh X$.
It follows from $\Min \shF \subseteq \Ass \shF \subseteq \Supp \shF$.
On the other hand, 
for a Serre subcategory $\catX$ of $\coh X$,
we have that $\Ass \catX= \Supp \catX$.
It is clear that $\Ass \catX \subseteq \Supp \catX$.
Take $x \in \Supp \catX$.
Then $\shO_{Z_x}\in \coh_{\Supp \catX} X =\catX$, and hence $x\in \Ass \catX$.

Since Serre subcategories of $\coh X$ are
torsionfree classes closed under tensoring with line bundles,
we obtain the following commutative diagram:
\[
\begin{tikzcd}
\{\text{torsionfree classes closed under tensoring with line bundles}\} \arrow[r,shift left,"\Ass"] & \{\text{arbitrary subsets}\} \arrow[l,shift left,"{\coh_{(-)}^{\ass}X}"]\\
\{\text{Serre subcategories}\} \arr{u,hook} \arrow[r,shift left,"\Supp"] & \{\text{specialization-closed subsets}\}. \arr{u,hook} \arrow[l,shift left,"{\coh_{(-)}X}"]
\end{tikzcd}
\]
\end{rmk}

We prove the lemmas needed to prove \cref{thm:Stanley-Wang for scheme} below.
\begin{lem}\label{lem:Z_x filtration}
Let $X$ be a divisorial noetherian scheme and $\shF$ a coherent $\shO_X$-module.
Let $\catX$ be an extension-closed subcategory of $\coh X$
closed under tensoring with line bundles.
If $\shO_{Z_x}\in\catX$ for any $x\in \Supp \shF$,
then $\shF \in \catX$.
\end{lem}
\begin{proof}
Suppose that $\shF \ne 0$ since the case where $\shF=0$ is obvious.
Take $x \in \Ass \shF\subseteq \Supp \shF$.
Then there is an injective $\shO_X$-linear map $i\colon \shO_{Z_{x}}\otimes \shL \inj \shF$
for some line bundle $\shL$ by \cref{prp:ass pt inj}.
Iterating this process for $\Cok(i)$ whenever it is nonzero, we find that
\[
\shF\in \gen{\shO_{Z_x}\otimes \shL \mid \text{$x\in \Supp \shF$, $\shL$ is a line bundle}}_{\ext} \subseteq \catX
\]
since $\shF$ is a noetherian object in $\Qcoh X$.
\end{proof}

\begin{lem}\label{lem:char x in Supp catX}
Let $X$ be a divisorial noetherian scheme
and $\catX$ a tensor-ideal subcategory of $\coh X$ 
closed under direct summands and extensions.
For any point $x\in X$,
we have that $x\in \Supp \catX$ if and only if $\shO_{Z_x} \in \catX$.
\end{lem}
\begin{proof}
The proof is inspired by \cite[Lemma 2.1]{IMST}.
It is clear that $\shO_{Z_x} \in \catX$ implies $x\in \Supp \catX$.
We prove the converse by induction on $\dim Z_x$.
Suppose that $x\in \Supp \catX$.
There exists $\shF \in \catX$ such that $x \in \Supp \shF$.
Then $\shG := \shF \otimes \shO_{Z_x}\in \catX$ since $\catX$ is tensor-ideal.
Let $i\colon Z_x \inj X$ be the canonical closed immersion.
Then $\shG \iso i_*i^*\shG$ and $\Supp \shG = \ol{\{x\}}$ hold.
If $\dim Z_x =0$, then $Z_x = \Spec \kk(x)$.
Thus, we have that $\shG \iso \shO_{Z_x}^{\oplus n}$ for some positive integer $n$,
and hence $\shO_{Z_x} \in \catX$ since $\catX$ is closed under direct summands.
Suppose that $\dim Z_x >0$.
We first note that $\shH \in \catX$ if $\Supp \shH \subsetneq \ol{\{x\}}$
for any $\shH \in \coh X$.
Indeed, for any $y\in \Supp \shH$, we have $\dim Z_y < \dim Z_x$,
and thus $\shO_{Z_y} \in \catX$ by the induction hypothesis.
Hence $\shH \in \catX$ by \cref{lem:Z_x filtration}.
Let $y_1,\dots, y_n$ be all the associated points of $\shG$ except $x$.
We have surjective morphisms $\shG \surj \shG_i$ for $0\le i \le n$
satisfying the following by \cref{fct:primary}:
\begin{itemize}
\item 
$\Ass \shG_0=\{x\}$ and $\Ass \shG_i=\{y_i\}$ hold for any $1\le i \le n$.
\item
The canonical morphism $\phi\colon \shG \to \bigoplus_{i=0}^n \shG_i$ is injective.
\end{itemize}
Then $\phi_x$ is surjective since $\shG_{i,x}=0$ for $1 \le i \le n$.
Thus $\Supp (\Cok \phi)$ is a proper closed subset of $\ol{\{x\}}$,
and $\Cok \phi$ belongs to $\catX$.
We obtain $\shG_0 \in \catX$ by the exact sequence
$0 \to \shG \xr{\phi} \bigoplus_{i=0}^n \shG_i \to \Cok \phi \to 0$.
Then $\shG_0 \iso i_*i^*\shG_0$ also holds since $\shG_0$ is a quotient of $\shG$.
From this, we have the following $\shO_{X,x}$-isomorphisms for some positive integer $n$:
\[
\shG_{0,x} \iso \left(i_*i^*\shG_0 \right)_x \iso \shG_{0,x} \otimes_{\shO_{X,x}} \kk(x) 
\iso \kk(x)^{\oplus n}.
\]
Then we have an $\shO_X$-linear map 
$\psi \colon  \shG_0 \to \shO_{Z_x}^{\oplus n} \otimes \shL$ 
for some line bundle $\shL$ such that $\psi_x$ is an isomorphism 
by \cref{lem:lift local hom to global hom}.
On the one hand,
we have $\Ker \psi =0$ because $x \not \in \Ass \Ker \psi \subseteq \Ass \shG_0 = \{x\}$.
On the other hand, we have $\Cok(\psi) \in \catX$ since
$x \not\in  \Supp(\Cok(\psi)) \subseteq \Supp \left(\shO_{Z_x}^{\oplus n} \otimes \shL \right)
=\ol{\{x\}}$.
Therefore, we obtain $\shO_{Z_x} \in \catX$ by the following exact sequence
since $\catX$ is a tensor-ideal subcategory closed under direct summands and extensions:
\[
0 \to \shG_0 \xr{\psi} \shO_{Z_x}^{\oplus n} \otimes \shL \to \Cok \psi \to 0.
\]
This proves the lemma by induction.
\end{proof}

\begin{lem}\label{lem:Supp catX inclusion}
Let $X$ be a divisorial noetherian scheme.
For any tensor-ideal subcategories $\catX$ and $\catY$ of $\coh X$ 
closed under direct summands and extensions,
we have that $\catX \subseteq \catY$ if and only if $\Supp\catX \subseteq \Supp\catY$.
In particular, we have that $\catX = \catY$ if and only if $\Supp\catX = \Supp\catY$.
\end{lem}
\begin{proof}
The only if part is obvious.
Suppose that $\Supp\catX \subseteq \Supp\catY$.
Let $\shF \in \catX$.
Then $\Supp \shF \subseteq \Supp \catX \subseteq \Supp \catY$ holds.
We obtain that $\shO_{Z_x} \in \catY$ for any $x\in \Supp \shF$
by \cref{lem:char x in Supp catX}.
This implies $\shF \in \catY$ by \cref{lem:Z_x filtration},
and hence $\catX \subseteq \catY$ holds.
\end{proof}

We now classify the torsion classes closed under tensoring with line bundles.
\begin{thm}\label{thm:Stanley-Wang for scheme}
Let $X$ be a divisorial noetherian scheme.
Then the following are equivalent for a subcategory $\catX$ of $\coh X$:
\begin{enur}
\item
$\catX$ is a Serre subcategory.
\item
$\catX$ is a torsion class closed under tensoring with line bundles.
\item
$\catX$ is a wide subcategory closed under tensoring with line bundles.
\item
$\catX$ is an ICE-closed subcategory closed under tensoring with line bundles.
\item
$\catX$ is a narrow subcategory closed under tensoring with line bundles.
\item
$\catX$ is a tensor-ideal subcategory closed under direct summands and extensions.
\end{enur}
\end{thm}
\begin{proof}
The implications (i)$\imply$(ii)$\imply$(iv)$\imply$(v)$\imply$(vi) 
and (i)$\imply$(iii)$\imply$(iv) easily follow from 
\cref{lem:char tors closed under tensoring,rmk:Serre L-closed}.
We now prove (vi)$\imply$(i).
Let $\catX$ be a tensor-ideal subcategory closed under direct summands and extensions.
Then we have $\Supp \catX = \Supp(\coh_{\Supp \catX} X)$ by \cref{fct:Gabriel}.
This implies $\catX = \coh_{\Supp \catX} X$ 
by the implication (i)$\imply$(vi) and \cref{lem:Supp catX inclusion}.
Thus $\catX$ is a Serre subcategory.
\end{proof}

\subsection{Classifying IE-closed subcategories closed under tensoring with line bundles}\label{ss:IE in coh}
Let $X$ be a divisorial noetherian scheme.
In this subsection,
we classify IE-closed subcategories of $\coh X$ closed under tensoring with line bundles.
More precisely,
we prove that for a subcategory of $\coh X$ closed under tensoring with line bundles,
it is a torsionfree class if and only if it is IE-closed (\cref{thm:IE for scheme}).
This generalizes Enomoto's classification of IE-closed subcategories of 
the category of finitely generated modules over a commutative noetherian ring 
(cf.\ \cref{prp:Enomoto's observation}).
This result plays a crucial role in the next section.

Let us begin with an easy observation.
\begin{lem}\label{lem:gen cat L-closed}
Let $X$ be a locally noetherian scheme.
If $\catX$ is a subcategory of $\coh X$ closed under tensoring with line bundles,
then so are $\gen{\catX}_{\sub}$, $\gen{\catX}_{\quot}$ and $\gen{\catX}_{\ext}$.
In particular, $\T(\catX)$ and $\F(\catX)$ are also closed under tensoring with line bundles.
\end{lem}
\begin{proof}
It follows from the constructions and the exactness of the functor $-\otimes \shL$.
\end{proof}

\begin{cor}
Let $X$ be a divisorial noetherian scheme and
$\catX$ a subcategory of $\coh X$
closed under tensoring with line bundles.
Then $\T(\catX)=\coh_{\Supp \catX} X$ and $\F(\catX)=\coh^{\ass}_{\Ass\catX} X$ hold.
\end{cor}
\begin{proof}
$\T(\catX)$ and $\F(\catX)$ are closed under tensoring with line bundles
by \cref{lem:gen cat L-closed}.
We have $\T(\catX)\subseteq \coh_{\Supp \catX} X$ by the minimality of $\T(\catX)$.
Then we obtain the following by \cref{fct:Gabriel}:
\[
\Supp(\coh_{\Supp \catX}X)= \Supp\catX 
\subseteq \Supp \T(\catX) \subseteq \Supp(\coh_{\Supp \catX} X)
\]
This implies $\T(\catX)=\coh_{\Supp \catX} X$ by \cref{lem:Supp catX inclusion}.
Similarly, we find that $\F(\catX)=\coh^{\ass}_{\Ass\catX} X$
by \cref{lem:Ass catX inclusion,thm:Takahashi for scheme}.
\end{proof}

We now classify the IE-closed subcategories closed under tensoring with line bundles.
\begin{thm}\label{thm:IE for scheme}
Let $X$ be a divisorial noetherian scheme.
The following are equivalent for a subcategory $\catX$ of $\coh X$
closed under tensoring with line bundles.
\begin{enur}
\item
$\catX$ is a torsionfree class.
\item
$\catX$ is an IKE-closed subcategory.
\item
$\catX$ is an IE-closed subcategory.
\end{enur}
\end{thm}
\begin{proof}
The proof is parallel to \cite[Theorem 3.1]{IE=torf}.
It is clear that (i)$\imply$(ii)$\imply$(iii).
Suppose that $\catX$ is an IE-closed subcategory closed under tensoring with line bundles.
Then $\catX=\T(\catX) \cap \F(\catX)$ by \cite[Theorem 2.7]{IE=torf}.
Since $\T(\catX)$ is a Serre subcategory of $\coh X$
by \cref{thm:Stanley-Wang for scheme,lem:gen cat L-closed},
$\catX=\T(\catX) \cap \F(\catX)$ is the intersection of the two torsionfree classes.
Thus $\catX$ is also a torsionfree class.
\end{proof}


\section{Classifying Serre subcategories in a torsionfree class}\label{s:Serre in torf}
Let $X$ be a noetherian scheme having an ample line bundle.
In this section,
we classify the Serre subcategories (cf.\ \cref{dfn:Serre in ex cat}) 
in a torsionfree class of $\coh X$ closed under tensoring with line bundles
(\cref{prp:classify Serre in torf}).
This gives a certain generalization of
Gabriel's classification of Serre subcategories (cf.\ \cref{fct:Gabriel}).

Let us begin with an observation for Serre subcategories in a torsionfree class.
\begin{lem}\label{lem:Serre in torf IKE}
Let $\catA$ be an abelian category
and $\catX$ its torsionfree class.
Then any Serre subcategory $\catS$ of $\catX$
is an IKE-closed subcategory of $\catA$.
\end{lem}
\begin{proof}
It is clear that $\catS$ is an extension-closed subcategory of $\catA$.
Let $f\colon M \to N$ be a morphism in $\catS$.
Since $\catX$ is a torsionfree class of $\catA$,
the canonical exact sequence $0\to \Ker(f) \to M \to \Ima(f) \to 0$ is a conflation of $\catX$.
From this and the fact that $\catS$ is a Serre subcategory of $\catX$,
we have  $\Ker(f), \Ima(f)\in\catS$.
This proves that $\catS$ is an IKE-closed subcategory of $\catA$.
\end{proof}

The following lemma states that 
if a torsionfree class possesses the property of being closed under tensoring with line bundles,
this property will be inherited by its Serre subcategories.
\begin{lem}\label{lem:Serre in torf ample stable}
Let $X$ be a noetherian scheme having an ample line bundle $\shL$.
Let $\catX$ be a torsionfree class of $\coh X$ closed under tensoring with line bundles,
and let $\catS$ be a Serre subcategory of $\catX$.
Then $\catS$ is also a torsionfree class of $\coh X$ closed under tensoring with line bundles.
\end{lem}
\begin{proof}
We first prove that $\catS$ is closed under tensoring with line bundles.
We may assume that $\shL$ is globally generated.
Since $\catS$ is an IKE-closed subcategory of $\coh X$ by \cref{lem:Serre in torf IKE},
it is enough to show that
$\shF\otimes \shL \in \catS$ for any $\shF\in\catS$
by \cref{lem:char torf closed under tensoring}.
Since $\shL$ is globally generated,
we have an exact sequence
\[
0 \to \shK \to \shO_X^{\oplus n} \to \shL \to 0.
\]
Applying the functor $\shF \otimes -$ to this sequence,
we get the exact sequence $\shK \otimes \shF \to \shF^{\oplus n} \xr{\phi} \shF\otimes \shL \to 0$.
Then the exact sequence 
$0 \to \Ker (\phi) \to \shF^{\oplus n} \to \shF\otimes \shL \to 0$
is a conflation in $\catX$ 
since it is a torsionfree class closed under tensoring with line bundles.
Then we have that $\shF\otimes \shL\in\catS$ 
because $\catS$ is a Serre subcategory in $\catX$ and $\shF^{\oplus n}\in\catS$.
Therefore
$\catS$ is an IKE-closed subcategory of $\coh X$ closed under tensoring with line bundles,
and hence it is a torsionfree class by \cref{thm:IE for scheme}.
This finishes the proof.
\end{proof}

We now classify the Serre subcategories 
in a torsionfree class closed under tensoring with line bundles.
\begin{thm}\label{prp:classify Serre in torf}
Let $X$ be a noetherian scheme having an ample line bundle.
Let $\catX$ be a torsionfree class of $\coh X$ closed under tensoring with line bundles.
\begin{enua}
\item
If $\Phi$ is a specialization-closed subset of $\Ass \catX$,
then $\coh_{\Phi}^{\ass} X$ is a Serre subcategory of $\catX$.
\item
If $\catS$ is a Serre subcategory of $\catX$,
then $\Ass\catS$ is a specialization-closed subset of $\Ass \catX$.
\item
The assignments $\catS \mapsto \Ass\catS$ and $\Phi \mapsto \coh^{\ass}_{\Phi} X$
give rise to mutually inverse bijections between the following sets:
\begin{itemize}
\item 
The set of Serre subcategories of $\catX$.
\item
The set of specialization-closed subsets of $\Ass \catX$
(with respect to the relative topology induced from $X$).
\end{itemize}
\end{enua}
\end{thm}
\begin{proof}
(1)
Let $0\to \shF \to \shG \to \shH \to 0$ be a conflation in $\catX$
such that $\Ass \shG \subseteq \Phi$.
Since $\coh_{\Phi}^{\ass} X$ is a torsionfree class,
it is enough to show that $\Ass \shH \subseteq \Phi$.
Take $x\in \Ass \shH$.
Since $\Ass \shH \subseteq \Supp \shH \subseteq \Supp \shG$,
there is $y\in \Min \shG$ such that
$y$ is a generalization of $x$.
Because $\Min\shG \subseteq \Ass \shG \subseteq \Phi$ 
and $\Phi$ is specialization-closed in $\Ass \catX$,
we have $x\in \Phi$.
Hence, we obtain $\Ass \shH \subseteq \Phi$.
This proves that $\coh_{\Phi}^{\ass} X$ is a Serre subcategory of $\catX$. 

(2)
Let $x\in \Ass \catS$ and $y\in \Ass\catX$ such that $y$ is a specialization of $x$.
Then $Z_y$ is a closed subscheme of $Z_x$.
Thus, we have the following exact sequence in $\coh X$:
\[
0 \to \shI \to \shO_{Z_x} \to \shO_{Z_y} \to 0.
\]
We can easily see that this sequence is a conflation of $\catX = \coh^{\ass}_{\Ass\catX} X$.
We have that $\shO_{Z_x}\in \catS$ by \cref{lem:char x in Ass,lem:Serre in torf ample stable}.
Then we obtain that $\shO_{Z_y} \in \catS$ since it is a Serre subcategory of $\catX$,
and hence we have that $y\in \Ass \catS$.
This proves that $\Ass\catS$ is a specialization-closed subset of $\Ass \catX$.

(3)
It follows from (1), (2), \cref{thm:Takahashi for scheme}
and \cref{lem:Serre in torf ample stable}.
\end{proof}

We will describe the applications of this theorem in the next section.

\section{Examples}\label{s:ex}

\subsection{The cases of quasi-affine schemes and commutative rings}\label{ss:quasi-affine}
In this subsection,
we describe the immediate consequences of the results 
in \S \ref{s:subcat in coh} and \ref{s:Serre in torf} 
for quasi-affine noetherian schemes and commutative noetherian rings.

Let us give a convenient way to determine
whether IE-closed subcategories are closed under tensoring with line bundles.
\begin{lem}\label{lem:char IE closed under tensoring}
Let $X$ be a noetherian scheme 
having an ample family $\{\shL_{\alpha}\}_{\alpha \in \LL}$.
The following are equivalent for an IE-closed subcategory $\catX$ of $\coh X$.
\begin{enur}
\item
$\catX$ is closed under tensoring with line bundles.
\item
Both $\shF\otimes \shL_{\alpha}$ and $\shF\otimes \shL_{\alpha}^{\vee}$ belong to $\catX$
for any $\shF\in\catX$ and any $\alpha \in \LL$.
\end{enur}
\end{lem}
\begin{proof}
We only prove (ii)$\imply$(i) since the converse is obvious.
We have $\catX=\T(\catX) \cap \F(\catX)$ by \cite[Theorem 2.7]{IE=torf}.
It is easy to see that
both $\T(\catX)$ and $\F(\catX)$ satisfy the condition (ii) by the constructions.
Hence, they are closed under tensoring with line bundles by 
\cref{lem:char torf closed under tensoring,lem:char tors closed under tensoring}.
In particular, so is $\catX=\T(\catX) \cap \F(\catX)$.
\end{proof}

\begin{cor}\label{cor:IEL-closed quasi-affine}
Let $X$ be a quasi-affine noetherian scheme.
Then every IE-closed subcategory is closed under tensoring with line bundles.
\end{cor}
\begin{proof}
It follows from \cref{lem:char IE closed under tensoring} and 
the fact that the structure sheaf $\shO_X$ is ample when $X$ is quasi-affine.
\end{proof}

As a result of this corollary, when $X$ is quasi-affine,
torsion(free) classes are automatically closed under tensoring with line bundles.
Therefore, we immediately obtain the following corollaries.

\begin{cor}\label{prp:torf in coh quasi-affine}
Let $X$ be a noetherian quasi-affine scheme.
Then the assignments $\Phi \mapsto \coh^{\ass}_{\Phi} X$ and $\catX \mapsto \Ass\catX$
give mutually inverse bijections between the following sets:
\begin{itemize}
\item
The power set of $X$.
\item 
The set of torsionfree classes of $\coh X$.
\end{itemize}
\end{cor}
\begin{proof}
It follows from \cref{thm:Takahashi for scheme,cor:IEL-closed quasi-affine}
\end{proof}

\begin{cor}\label{prp:tors in coh quasi-affine}
Let $X$ be a noetherian quasi-affine scheme.
The following are equivalent for a subcategory $\catX$ of $\coh X$.
\begin{enur}
\item
$\catX$ is a Serre subcategory.
\item
$\catX$ is a torsion class.
\item
$\catX$ is a wide subcategory.
\item
$\catX$ is an ICE-closed subcategory.
\item
$\catX$ is a narrow subcategory.
\item
$\catX$ is a tensor-ideal subcategory closed under direct summands and extensions.
\end{enur}
\end{cor}
\begin{proof}
It follows from \cref{lem:char tors closed under tensoring,thm:Stanley-Wang for scheme}.
\end{proof}

\begin{cor}\label{prp:IE in coh quasi-affine}
Let $X$ be a noetherian quasi-affine scheme.
The following are equivalent for a subcategory $\catX$ of $\coh X$.
\begin{enur}
\item
$\catX$ is a torsionfree class.
\item
$\catX$ is an IKE-closed subcategory.
\item
$\catX$ is an IE-closed subcategory.
\end{enur}
\end{cor}
\begin{proof}
It follows from \cref{thm:IE for scheme,cor:IEL-closed quasi-affine}.
\end{proof}

\begin{cor}\label{prp:Serre in torf quasi-affine}
Let $X$ be a noetherian quasi-affine scheme and $\catX$ a torsionfree class of $\coh X$.
Then the assignments $\catS \mapsto \Ass\catS$ and $\Phi \mapsto \coh^{\ass}_{\Phi} X$
give rise to mutually inverse bijections between the following sets:
\begin{itemize}
\item 
The set of Serre subcategories of $\catX$.
\item
The set of specialization-closed subsets of $\Ass \catX$.
\end{itemize}
\end{cor}
\begin{proof}
It follows from \cref{prp:classify Serre in torf,cor:IEL-closed quasi-affine}.
\end{proof}

Therefore, we obtain the following implications for subcategories of $\coh X$,
where $X$ is a quasi-affine noetherian scheme:
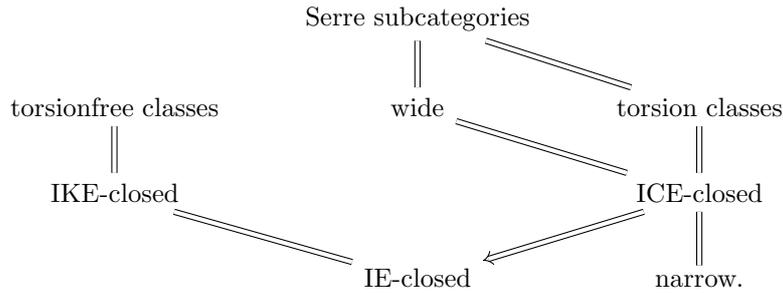
\[
\begin{tikzcd}
&\text{Serre subcategories} \arr{d,equal} \arr{rd,equal}& \\
\text{torsionfree classes} \arr{d,equal} & \text{wide}  \arr{rd,equal} & \text{torsion classes} \arr{d,equal}\\
\text{IKE-closed} \arr{rd,equal} && \text{ICE-closed} \arr{ld,Rightarrow} \arr{d,equal}\\
& \text{IE-closed} & \text{narrow}.
\end{tikzcd}
\]

Applying these results to a noetherian affine scheme $X=\Spec R$,
we can recover the known results for $\catmod R$.

\begin{cor}[{\cite[Theorem 4.1]{Takahashi}}]\label{fct:Takahashi torf}
Let $R$ be a commutative noetherian ring.
Then the assignments
\[
\catX \mapsto \Ass\catX :=\bigcup_{M\in\catX} \Ass M
\quad \text{and} \quad
\Phi \mapsto \catmod^{\ass}_{\Phi} R:=\{M\in \catmod R \mid \Ass M\subseteq \Phi\}
\]
give rise to mutually inverse bijections between the following sets:
\begin{itemize}
\item 
The set of torsionfree classes of $\catmod R$.
\item
The power set of $\Spec R$.
\end{itemize}
\end{cor}
\begin{proof}
It follows from \cref{prp:torf in coh quasi-affine}.
\end{proof}


\begin{cor}[{\cite[Corollary 7.1]{SW}, \cite[Theorem 2.5]{IMST}}]\label{fct:tor=Serre}
Let $R$ be a commutative noetherian ring.
The following are equivalent for a subcategory $\catX$ of $\catmod R$.
\begin{enur}
\item
$\catX$ is a Serre subcategory.
\item
$\catX$ is a torsion class.
\item
$\catX$ is a wide subcategory.
\item
$\catX$ is an ICE-closed subcategory.
\item
$\catX$ is a narrow subcategory.
\item
$\catX$ is a tensor-ideal subcategory closed under direct summands and extensions.
\end{enur}
\end{cor}
\begin{proof}
It follows from \cref{prp:tors in coh quasi-affine}.
\end{proof}

\begin{cor}[{\cite[Theorem 3.1.]{IE=torf}}]\label{prp:Enomoto's observation}
Let $R$ be a commutative noetherian ring.
The following are equivalent for a subcategory $\catX$ of $\catmod R$.
\begin{enur}
\item
$\catX$ is a torsionfree class.
\item
$\catX$ is an IKE-closed subcategory.
\item
$\catX$ is an IE-closed subcategory.
\end{enur}
\end{cor}
\begin{proof}
It follows from \cref{prp:IE in coh quasi-affine}.
\end{proof}

\cref{prp:classify Serre in torf} 
remains new and interesting even for commutative noetherian rings.

\begin{cor}\label{prp:Serre in torf mod R}
Let $R$ be a commutative noetherian ring and $\catX$ a torsionfree class of $\catmod R$.
Then the assignments $\catS \mapsto \Ass\catS$ and $\Phi \mapsto \catmod^{\ass}_{\Phi} R$
give rise to mutually inverse bijections between the following sets:
\begin{itemize}
\item 
The set of Serre subcategories of $\catX$.
\item
The set of specialization-closed subsets of $\Ass \catX$.
\end{itemize}
\end{cor}
\begin{proof}
It follows from \cref{prp:Serre in torf quasi-affine}.
\end{proof}

\subsection{Classifying Serre subcategories of the category of maximal pure sheaves}\label{ss:Serre in pure}
Let $X$ be a noetherian scheme.
Consider the following subcategories of $\coh X$:
\begin{itemize}
\item
The category $\tf X$ of locally torsionfree coherent $\shO_X$-modules.
\item
The category $\pure X$ of maximal pure coherent $\shO_X$-modules
(when $\dim X < \infty$).
\item
The category $\cm X$ of maximal Cohen-Macaulay coherent $\shO_X$-modules.
\item
The category $\vect X$ of locally free $\shO_X$-modules of finite rank.
\end{itemize}
For a noetherian commutative ring $R$,
we write $\tf R := \tf(\Spec R)$, $\pure R:=\pure(\Spec R)$, and $\cm R:=\cm(\Spec R)$.
In this subsection,
we will classify Serre subcategories of these categories
by using \cref{prp:classify Serre in torf}.
We refer the reader to Appendix \ref{s:basic on torf sh} 
for the notion of locally torsionfree sheaves, pure sheaves and Cohen-Macaulay sheaves.
We first describe the relationship between these categories.
\begin{itemize}
\item
Suppose that $\dim X < \infty$.
In general, we have $\pure X \subseteq \tf X$ (cf.\ \cref{prp:maximal pure Ass}).
The equality $\pure X = \tf X$ holds when $X$ is equidimensional and has no embedding points
(cf.\ \cref{cor:pure and torf}).
\item
In general, we have $\cm X \subseteq \tf X$ (cf.\ \cref{prp:MCM and torf}).
The equality $\cm X = \tf X$ holds when $X$ is Cohen-Macaulay scheme and $\dim X \le 1$
(cf.\ \cref{prp:torf on 1-dim CM}).
\item
If $X$ is Cohen-Macaulay, then we have $\vect X \subseteq \cm X$.
The equality $\vect X = \cm X$ holds when $X$ is regular
(cf.\ \cref{prp:torf on regular}).
\item
In particular,
if $X$ is equidimensional of $\dim X \le 1$ and Cohen-Macaulay,
then we have $\pure X = \tf X = \cm X$.
\item
In particular,
if $X$ is regular and $\dim X \le 1$,
then we have $\tf X = \cm X = \vect X$.
\end{itemize}

The following classifications of Serre subcategories are immediately obtained 
from the bijections in \cref{prp:classify Serre in torf}.
\begin{cor}\label{prp:Serre in tf and pure}
Let $X$ be a noetherian scheme having an ample line bundle.
\begin{enua}
\item
There is a bijection between the following two sets:
\begin{itemize}
\item 
The set of Serre subcategories of $\tf X$.
\item
The set of specialization-closed subsets of $\Ass X$.
\end{itemize}
\item
If $\dim X < \infty$,
then there is a bijection between the following two sets:
\begin{itemize}
\item 
The set of Serre subcategories of $\pure X$.
\item
The power set of the set of irreducible components $Z$ of $X$ such that $\dim Z = \dim X$.
\end{itemize}
\end{enua}
\end{cor}
\begin{proof}
We have $\tf X = \coh^{\ass}_{\Ass X} X$ and $\pure X = \coh^{\ass}_{\Assh X} X$
by \cref{prp:torf sh general,prp:maximal pure Ass}, respectively.
Thus (1) follows from \cref{prp:classify Serre in torf}.
Since any subsets of $\Assh X$ are specialization-closed,
the Serre subcategories of $\pure X$ bijectively correspond to
the subsets of $\Assh X$ by \cref{prp:classify Serre in torf}.
Because the elements of $\Assh X$ bijectively correspond to 
the irreducible components $Z$ of $X$ such that $\dim Z = \dim X$,
we obtain (2).
\end{proof}

\begin{cor}
Let $X$ be a noetherian scheme having an ample line bundle.
\begin{enua}
\item
$\tf X$ has only finitely many Serre subcategories.
\item
If $\dim X < \infty$,
then $\pure X$ has only finitely many Serre subcategories.
\end{enua}
\end{cor}
\begin{proof}
It follows from \cref{prp:Serre in tf and pure} and the fact that $\Ass X$ is a finite set.
\end{proof}

\begin{ex}
Consider the ring $R:=\bbk[x,y]/(x^2,xy)$,
where $\bbk[x,y]$ is the polynomial ring in two variables over a field $\bbk$.
Then we have $\Ass R=\{(x), (x,y)\}$ and $\Assh R = \{(x)\}$.
Thus $\tf R$ has exactly three Serre subcategories, as follows:
\[
\catmod^{\ass}_{\emptyset} R=0,\quad
\catmod^{\ass}_{\{(x,y)\}} R=\fl R:=\{\text{$R$-modules of finite length}\} ,\quad
\catmod^{\ass}_{\Ass R} R=\tf R.
\]
On the other hand, $\pure R$ has no nontrivial Serre subcategories.
\end{ex}


\cref{prp:Serre in tf and pure} (1) can be simplified if $X$ has no embedding points.
\begin{cor}\label{prp:Serre in tf without emb pt}
Let $X$ be a noetherian scheme without embedding points.
Suppose that $X$ has an ample line bundle.
Then there is a bijection between the following two sets:
\begin{itemize}
\item 
The set of Serre subcategories of $\tf X$.
\item
The power set of the set of irreducible components of $X$.
\end{itemize}
\end{cor}
\begin{proof}
Since any subsets of $\Ass X=\Min X$ are specialization-closed,
the Serre subcategories of $\tf X$ bijectively correspond to the subsets of $\Min X$
by \cref{prp:Serre in tf and pure}.
Because the elements of $\Min X$ bijectively correspond to 
the irreducible components of $X$, we obtain the corollary.
\end{proof}

As special cases of \cref{prp:Serre in tf without emb pt},
we obtain the classifications of the Serre subcategories of $\cm X$ and $\vect X$.
\begin{cor}\label{prp:Serre in cm}
Let $X$ be a noetherian Cohen-Macaulay scheme such that $\dim X \le 1$.
Suppose that $X$ has an ample line bundle.
Then there is a bijection between the following two sets:
\begin{itemize}
\item 
The set of Serre subcategories of $\cm X$.
\item
The power set of the set of irreducible components of $X$.
\end{itemize}
\end{cor}
\begin{proof}
It follows from $\cm X = \tf X$ and \cref{prp:Serre in tf without emb pt}.
\end{proof}

In particular, the following holds in the affine case.
\begin{cor}\label{prp:Serre in cm aff}
Let $R$ be a $1$-dimensional noetherian Cohen-Macaulay commutative ring.
Then there is a bijection between the following two sets:
\begin{itemize}
\item 
The set of Serre subcategories of $\cm R$.
\item
The power set of $\Min R$.
\qed
\end{itemize}
\end{cor}

\begin{cor}\label{prp:Serre in vect}
Let $X$ be a noetherian regular scheme such that $\dim X \le 1$.
Suppose that $X$ has an ample line bundle.
Then there is a bijection between the following two sets:
\begin{itemize}
\item 
The set of Serre subcategories of $\vect X$.
\item
The power set of the set of connected components of $X$.
\end{itemize}
\end{cor}
\begin{proof}
It follows from $\vect X = \tf X$, \cref{prp:Serre in tf without emb pt}
and the fact that every irreducible component of a regular scheme 
is a connected component.
\end{proof}

The following corollary is the author's original motivation for this paper.
\begin{cor}\label{prp:Serre in pure of reduced curve}
Let $Z=\bigcup_{i=1}^r C_i$ be a reduced projective curve over a field $\bbk$ 
with irreducible components $C_i$.
Then there is a bijection between the following two sets:
\begin{itemize}
\item 
The set of Serre subcategories of $\pure X = \tf X = \cm X$.
\item
The power set of $\{C_1,\dots,C_n\}$.
\end{itemize}
\end{cor}
\begin{proof}
A curve is an algebraic $\bbk$-scheme that is equidimensional of dimension $1$ by definition.
Thus, $Z$ is an equidimensional Cohen-Macaulay scheme of dimension $1$.
Then $\pure X = \tf X = \cm X$ hold, and the assertion follows from \cref{prp:Serre in cm}.
\end{proof}

\cref{prp:Serre in pure of reduced curve} immediately yields the following.
\begin{cor}[{\cite[Corollary 4.10]{Saito}}]\label{prp:Serre in conn sm proj}
Let $C$ be a connected smooth projective curve over a field $\bbk$.
Then $\vect C$ has no nontrivial Serre subcategories.
\qed
\end{cor}
\subsection{Classifying torisionfree classes of the projective line}
In this subsection,
we classify the \emph{arbitrary} torsionfree classes of 
the category $\coh \bbP^1$ of coherent sheaves on the projective line.
This classification explains the necessity of 
the assumption of being closed under tensoring with line bundles in \S \ref{s:subcat in coh}.
Indeed, we will see that there exist distinct torsionfree classes $\catX$ and $\catY$ 
of $\coh \bbP^1$ such that $\Ass \catX = \Ass \catY$
(cf.\ \cref{prp:torf in coh P^1}).
For future work, we discuss torsionfree classes of the category $\coh C$
of coherent sheaves on a connected smooth projective curve.
We refer to \cite{Atiyah,Hartshorne,LeP} for the basics of coherent sheaves 
on a connected smooth projective curve.
See also \cite[Section 4]{Saito}.

We recall some basic notions for coherent sheaves on a curve to fix the notation.
Let $C$ be a connected smooth projective curve over a field $\bbk$
with generic point $\eta$.
We denote by $C_0$ the set of closed points of $C$.
Define $\tor C := \coh_{C_0} C$ (see \S \ref{ss:tor in coh} for the notation)
and call its objects \emph{torsion sheaves}.
Then $(\tor C, \vect C)$ is a torsion pair of $\coh C$
(see, for example, \cite[Lemma 4.11]{Saito}), that is, the following hold:
\begin{enur}
\item
$\Hom(\shF,\shE)=0$ for any $\shF\in \tor C$ and $\shE \in \vect C$.
\item
$\coh C = \tor C * \vect C$ holds (see \cref{dfn:notation subcat} for the notation).
\end{enur}
Thus, for any $\shF \in \coh C$, we have an exact sequence
\[
0 \to \shF_{\rm{tor}} \to \shF \to \shF_{\rm{vect}} \to 0,
\]
where $\shF_{\rm{tor}}\in \tor C$ and $\shF_{\rm{vect}}\in \vect C$.
Moreover, this exact sequence splits by the Serre duality:
\[
\Ext^1(\shF_{\rm{vect}},\shF_{\rm{tor}})
\iso \Hom(\shF_{\rm{tor}},\shF_{\rm{vect}}\otimes \omega_C)=0.
\]
Here $\omega_C$ is the canonical line bundle on $C$.

For a divisor $D$ on $C$, we denote by $\shO_C(D)$ the line bundle associated with $D$.
Any line bundles are of the form $\shO_C(D)$ for some divisor $D$.
We write $\shF(D):= \shF \otimes \shO_C(D)$ for $\shF \in \coh C$.
For any closed point $x\in C$,
we set $\shO_{n{\cdot}x}:=i_*\left(\shO_{C,x}/\mm_x^n\right)$,
where $i$ is the natural morphism $\Spec \shO_{C,x} \to C$.
When a divisor $D=\sum_{i=1}^r n_i{\cdot}x_i$ is effective,
we define the torsion sheaf $\shO_D$ associated with $D$ by
\[
\shO_D := \bigoplus_{i=1}^r \shO_{n_i{\cdot}x_i} \in \tor C.
\]
Note that any torsion sheaf is of the form $\shO_D$ for some effective divisor $D$
(see, for example, \cite[Lemma 4.3]{Saito}).

We first determine
when a torsionfree class is closed under tensoring with line bundles.
For this, let us begin with observations for a torsionfree class of $\coh C$.
\begin{lem}\label{prp:property of torf in coh C}
Let $C$ be a connected smooth projective curve over a field $\bbk$
and $\catX$ a torsionfree class of $\coh C$.
\begin{enua}
\item
If $\shF \in \catX$, then both $\shF_{\rm{vect}}$ and $\shF_{\rm{tor}}$ belong to $\catX$.
\item
If $\catX$ contains a vector bundle $\shE$,
then $\shE(-D) \in \catX$ for any effective divisor $D$.
\item
If $\catX$ contains a nonzero vector bundle $\shE$ and is closed under tensoring with line bundles,
then $\vect C \subseteq \catX$ holds.
\item
If $\catX$ contains a nonzero vector bundle and a nonzero torsion sheaf,
then $\vect C \subseteq \catX$ holds.
\end{enua}
\end{lem}
\begin{proof}
(1)
It follows from $\shF \iso \shF_{\rm{tor}}\oplus \shF_{\rm{vect}}$ 
and the fact that $\catX$ is closed under direct summands.

(2)
It follows from the exact sequence $0 \to \shE(-D) \to \shE \to \shO_D^{\oplus \rk \shE} \to 0$.

(3)
We first prove that any line bundles belong to $\catX$.
It is clear when $\rk \shE=1$ since $\catX$ is closed under tensoring with line bundles.
Suppose that $r:=\rk \shE >1$.
Let $\shO_C(1)$ be a very ample sheaf on $C$.
Then there exists the following exact sequence for a sufficiently large $n\gg0$
by \cite[Theorem 2]{Atiyah}:
\[
0 \to \shO_C^{\oplus(r-1)} \to \shE(n) \to \shL \to 0,
\]
where $\shL$ is a line bundle.
Then $\shO_C \in \catX$ since $\catX$ is closed under subobjects and tensoring with line bundles.
Thus any line bundles belong to $\catX$.
Next, we prove that any vector bundles belong to $\catX$.
Let $\shV \in \vect C$.
There exists the following exact sequence for a sufficient large $n\gg0$
again by \cite[Theorem 2]{Atiyah}:
\[
0 \to \shO_C(-n)^{\oplus(r-1)} \to \shV \to \shL \to 0,
\]
where $\shL$ is a line bundle.
Then $\shV \in\catX$ since both $\shO_C(-n)$ and $\shL$ belong to $\catX$.

(4)
Let $\shE \in \catX$ be a nonzero vector bundle
and $D$ the effective divisor corresponding to a nonzero torsion sheaf $\shF \in \catX$.
We first claim that $\shE(n D) \in \catX$ for any $n\ge 0$.
It follows from an exact sequence $0 \to \shE \to \shE(n D) \to \shO_{nD}^{\oplus \rk \shE} \to 0$
and $\shO_{nD}\iso \shF^{\oplus n}$.
Next, we claim that $\shE(E) \in \catX$ for any divisor $E$.
Indeed, there is some positive integer $n>0$ such that $nD \ge E$.
Thus, we have that $\shE(E)= \shE(nD -(nD-E)) \in \catX$ by (2).
This proves that $\vect C \cap \catX$ is closed under tensoring with line bundles.
Therefore, we have $\vect C = \vect C \cap \catX \subseteq \catX$ by (3).
\end{proof}

We now determine when a torsionfree class of $\coh C$ is closed under tensoring with line bundles.
\begin{lem}\label{prp:torf in coh C}
Let $C$ be a connected smooth projective curve over a field $\bbk$
and $\eta$ its generic point.
Let $\catX$ be a torsionfree class of $\coh C$ and set $\Phi := \Ass \catX$.
\begin{enua}
\item
If $\Phi \subseteq C_0$, 
then $\catX \subseteq \tor C$ and $\catX$ is closed under tensoring with line bundles.
In this case, 
we have $\catX = \add(\shO_{m{\cdot}x} \mid x\in \Phi, m>0)$.
\item
If $\Phi \supsetneq \{\eta\}$,
then $\vect C \subseteq \catX$ and $\catX$ is closed under tensoring with line bundles.
In this case, 
we have $\catX = \add(\shE, \shO_{m{\cdot}x} \mid \shE \in \vect C, x\in \Phi_0, m>0)$,
where $\Phi_0:= \Phi \cap C_0$.
\item
If $\Phi = \{ \eta \}$,
then $\catX \subseteq \vect C$.
In this case, $\catX$ is closed under tensoring with line bundles
if and only if either $\catX = \vect C$ or $\catX= 0$.
\end{enua}
\end{lem}
\begin{proof}
(1)
It is clear that $\catX \subseteq \coh_{C_0}^{\ass} C =\tor C$.
It follows from $\shF \otimes \shL \iso \shF$ for any torsion sheaf $\shF$ and line bundle $\shL$
that $\catX$ is closed under tensoring with line bundles.
Then we can easily see that 
$\catX = \coh^{\ass}_{\Phi} C = \add(\shO_{m{\cdot}x} \mid x\in \Phi, m>0)$.

(2)
We can easily see that $\catX$ contains $\vect C$ 
by \cref{prp:property of torf in coh C} (1) and (4).
Let $\shF \in \catX$ and $\shL$ a line bundle.
We obtain an isomorphism
\[
\shF \otimes \shL 
\iso \left(\shF_{\rm{tor}} \oplus \shF_{\rm{vect}}\right)\otimes \shL
\iso \shF_{\rm{tor}} \oplus \left(\shF_{\rm{vect}}\otimes \shL\right).
\]
Then $\shF \otimes \shL$ belongs to $\catX$
since $\shF_{\rm{tor}} \in \catX$ and $\shF_{\rm{vect}}\otimes \shL \in \vect C \subseteq \catX$.
Thus $\catX$ is closed under tensoring with line bundles.
Then we can easily see that 
$\catX = \coh^{\ass}_{\Phi} C
= \add(\shE, \shO_{m{\cdot}x} \mid \shE \in \vect C, x\in \Phi_0, m>0)$.

(3)
It is clear that $\catX \subseteq \coh_{\{\eta\}}^{\ass} C =\vect C$.
Suppose that $\catX$ is closed under tensoring with line bundles and $\catX\ne 0$.
Then we obtain $\vect C \subseteq \catX$ by \cref{prp:property of torf in coh C} (3),
and hence $\catX= \vect C$.
The converse is obvious.
\end{proof}

Conversely,
for any subset $\Phi_0 \subseteq C_0$,
it is clear that
$\coh_{\Phi_0}^{\ass} C = \add(\shO_{m{\cdot}x} \mid x\in \Phi_0, m>0)$ and 
$\coh_{\Phi_0\cup\{\eta\}}^{\ass} C = \add(\shE, \shO_{m{\cdot}x} \mid \shE\in\vect C, x\in \Phi_0, m>0)$ are torsionfree classes of $\coh C$.
Thus, we obtain the following proposition.
\begin{prp}\label{prp:list torf in coh}
Let $C$ be a connected smooth projective curve over a field $\bbk$.
The complete list of torsionfree classes $\catX$ of $\coh C$ is the following:
\begin{enur}
\item
$\catX = \add(\shO_{m{\cdot}x} \mid x\in \Phi_0, m>0)$ 
for a subset $\Phi_0 \subseteq C_0$.
In this case, $\catX$ is closed under tensoring with line bundles
and $\Ass \catX = \Phi_0$.
\item
$\catX = \add(\shE, \shO_{m{\cdot}x} \mid \shE\in\vect C, x\in \Phi_0, m>0)$ 
for a subset $\Phi_0 \subseteq C_0$.
In this case, $\catX$ is closed under tensoring with line bundles
and $\Ass \catX = \{\eta\}\cup \Phi_0$.
\item
$0 \subsetneq \catX \subsetneq \vect C$.
In this case, $\catX$ is \emph{not} closed under tensoring with line bundles
and $\Ass \catX = \{\eta\}$.
\qed
\end{enur}
\end{prp}

Therefore, to classify torsionfree classes of $\coh C$,
it is sufficient to classify torsionfree classes of $\coh C$ contained in $\vect C$.

We now come back to the case of the projective line $\bbP^1$.
Let us classify the torsionfree classes of $\coh \bbP^1$.
Recall that any indecomposable vector bundle on $\bbP^1$ is a line bundle,
and it is of the form $\shO(n)$ for $n\in \bbZ$.
\begin{cor}\label{prp:torf in coh P^1}
The complete list of torsionfree classes $\catX$ of $\coh \bbP^1$ is the following:
\begin{enur}
\item
$\catX = \add(\shO_{m{\cdot}x} \mid x\in \Phi_0, m>0)$ 
for a subset $\Phi_0 \subseteq (\bbP^1)_0$.
In this case, $\catX$ is closed under tensoring with line bundles
and $\Ass \catX = \Phi_0$.
\item
$\catX = \add(\shO(i), \shO_{m{\cdot}x} \mid i\in\bbZ, x\in \Phi_0, m>0)$ 
for a subset $\Phi_0 \subseteq (\bbP^1)_0$.
In this case, $\catX$ is closed under tensoring with line bundles
and $\Ass \catX = \{\eta\}\cup\Phi_0$.
\item
$\catX = \add(\shO(i) \mid i \le n)$ for an integer $n\in\bbZ$.
In this case, $\catX$ is \emph{not} closed under tensoring with line bundles
and $\Ass \catX = \{\eta\}$.
\end{enur}
\end{cor}
\begin{proof}
It is enough to consider the case $0 \subsetneq \catX \subsetneq \vect \bbP^1$
by \cref{prp:list torf in coh}.
Set $n:= \max\{i \in \bbZ \mid \shO(i) \in \catX\}$.
Then $n<\infty$.
Indeed, if $n=\infty$, 
then we have $\catX = \vect \bbP^1$ by \cref{prp:property of torf in coh C} (2).
Thus, we have $\catX = \add(\shO(i) \mid i \le n)$ 
again by \cref{prp:property of torf in coh C} (2).
Conversely, 
it is clear that $\add(\shO(i) \mid i \le n)$ is a torsionfree class for any $n\in\bbZ$.
\end{proof}

\begin{rmk}
Let us recall the relationship between torsion pairs and torsionfree classes.
Let $\catA$ be an abelian category.
If $(\catT,\catF)$ is a torsion pair of $\catA$,
then $\catF$ is a torsionfree class.
However, even if $\catX$ is a torsionfree class,
$({}^{\perp}\catX,\catX)$ is not necessarily a torsion pair.
Here ${}^{\perp}\catX$ is the subcategory consisting of $A \in \catA$
such that $\Hom_{\catA}(A,X)=0$ for any $X\in\catX$.
In general, for a torsionfree class $\catX$ of $\catA$,
the following three conditions are equivalent:
\begin{enua}
\item
$({}^{\perp}\catX,\catX)$ is a torsion pair.
\item
$\catX$ is reflexive,
that is, the inclusion functor $\catX \inj \catA$ has a left adjoint.
\item
$\catX$ is covariantly finite,
that is,
for any $A \in\catA$, there is a morphism $f\colon A \to X_A$ to an object of $\catX$
such that $\Hom_{\catA}(f,X) \colon \Hom_{\catA}(X_A,X) \to \Hom_{\catA}(A,X)$ is surjective 
for any $X\in\catX$.
\end{enua}
It is immediate that (1) $\equi$ (2).
We refer to \cite[Theorem 3.2]{CCS} for (2) $\equi$ (3).

The torsion pairs $(\catT,\catF)$ of $\coh \bbP^1$ are classified 
by Chen, Lin and Ruan in \cite[Section 6]{CLR}, as follows:
\begin{enur}
\item
The trivial torsion pairs $(\coh \bbP^1, 0)$ and $(0,\coh \bbP^1)$.
\item
$\catT=\add(\shO_{m{\cdot}x} \mid x \not\in \Phi_0, m> 0)$ and 
$\catF=\add(\shO(i), \shO_{m{\cdot}x} \mid i\in\bbZ, x\in \Phi_0, m>0)$
for a proper subset $\Phi_0 \subseteq \left(\bbP^1\right)_0$.
\item
$\catT=\add(\shO(i), \shO_{m{\cdot}x} \mid i>n, x\in (\bbP^1)_0, m>0)$ and 
$\catF=\add(\shO(i) \mid i\le n)$
for an integer $n\in \bbZ$.
\end{enur}
Comparing this classification with \cref{prp:torf in coh P^1},
we can see that the difference between them is 
the torsionfree classes in (i) of \cref{prp:torf in coh P^1}.
More precisely,
we find that a nonzero torsionfree class $\catX$ of $\coh \bbP^1$
is not associated with a torsion pair
if and only if it is contained in $\tor C$.
\end{rmk}
\appendix
\def\thesection{\Alph{section}}
\section{Basics of torsionfree sheaves}\label{s:basic on torf sh}
In this appendix,
we give a systematic treatment for torsionfree sheaves and related concepts
for the reader's convenience.

\subsection{Associated points}\label{ss:Ass}
In this subsection,
we gather the facts about associated points used in this paper.
The main references are \cite{Matsumura,EGA4-2}.

We first recall the notion of associated prime ideals.
Let $R$ be a commutative ring and $M$ an $R$-module.
A prime ideal $\pp$ of $R$ is called \emph{associated prime} of $M$
if there is an injective $R$-linear map $R/\pp \inj M$.
We denote by $\Ass M := \Ass_R M$ the set of associated prime ideals of $M$.
We denote by $\Min M := \Min_R M$ the set of minimal elements of $\Ass M$
with respect to the inclusion-order.
An element of $\Min M$ is called an \emph{isolated prime ideal}.
An associated prime ideal which is not isolated 
is called an \emph{embedding prime ideal}.
Then $M$ has no embedding prime ideals if and only if $\Ass M = \Min M$.

Next, we recall the notion of associated points,
which generalize that of associated prime ideals to schemes.
Let $X$ be a scheme and $\shF\in \Qcoh X$.
A point $x\in X$ is called an \emph{associated point} of $\shF$
if $\mm_x \in \Ass_{\shO_{X,x}}(\shF_x)$,
that is, there is an injective $\shO_{X,x}$-linear map $\kk(x) \inj \shF_x$.
We denote by $\Ass \shF := \Ass_X \shF$ the set of associated points of $\shF$.
An \emph{isolated point} of $\shF$ is a maximal element of $\Ass \shF$ 
with respect to the specialization-order.
An \emph{embedding point} of $\shF$ is an associated point of $\shF$
which is not isolated.
We denote by $\Min \shF := \Min_X \shF$ the set of isolated points.
Then $\shF$ has no embedding point if and only if $\Ass \shF = \Min \shF$.
We call an associated point of $\shO_X$ an associated point of $X$.
We write $\Ass X := \Ass \shO_X$ and $\Min X:= \Min \shO_X$.

Let us describe the properties of associated points that hold 
for general schemes and quasi-coherent sheaves.
\begin{fct}[{cf.\ \cite[The text following (3.1.1) and (3.1.7)]{EGA4-2}}]\label{fct:Ass general scheme}
Let $X$ be a scheme and $\shF\in \Qcoh X$.
\begin{enua}
\item
$\Ass(\shF) \subseteq \Supp(\shF)$ holds.
\item
$\Ass_X(\shF)\cap U = \Ass_{U}(\shF_{|U})$
for any open subset $U$ of $X$.
\item
For any exact sequence $0 \to \shF_1 \to \shF_2 \to \shF_3 \to 0$ in $\Qcoh X$,
we have that $\Ass(\shF_1) \subseteq \Ass(\shF_2) \subseteq \Ass(\shF_1) \cup \Ass(\shF_3)$.
\end{enua}
\end{fct}

The following proposition asserts
that the property of having no embedding points is local.
\begin{prp}
Let $X$ be a scheme and $\shF\in \Qcoh X$.
\begin{enua}
\item
$\Min_X(\shF)\cap U = \Min_{U}(\shF_{|U})$
for any open subset $U$ of $X$.
\item
The following are equivalent:
\begin{itemize}
\item
$\shF$ has no embedding points.
\item
$\shF_{|U}$ has no embedding points for any open subset $U$ of $X$.
\item
There is an open covering $\{U_i\}_{i\in I}$ of $X$
such that $\shF_{|U_i}$ has no embedding points for any $i\in I$.
\end{itemize}
\end{enua}
\end{prp}
\begin{proof}
(1) follows from
\cref{fct:Ass general scheme} (2) and the fact that open subsets are generalization-closed.
(2) follows from (1) and \cref{fct:Ass general scheme} (2).
\end{proof}

The notion of associated points is of particular interest 
in the case of locally noetherian schemes. 
\begin{fct}[{cf.\ \cite[(3.1.2), (3.1.4), (3.1.5), (3.1.10)]{EGA4-2} and \cite[Theorem 6.2]{Matsumura}}]\label{fct:Ass loc noetherian}
Let $X$ be a locally noetherian scheme and $\shF\in \Qcoh X$.
\begin{enua}
\item
$\Ass_X(\shF)\cap U = \Ass_{\shO_X(U)} \shF(U)$
for any affine open subset $U$ of $X$.
\item
$\Ass_X (\shF) \cap \Spec \shO_{X,x} = \Ass_{\shO_{X,x}}(\shF_x)$
for any point $x\in X$.
\item
$\Min \shF$ coincides with 
the set of maximal elements of $\Supp \shF$ with respect to the specialization-order.
\item
$\shF=0$ if and only if $\Ass\shF=\emptyset$.
\item
For any closed immersion $j\colon Z \inj X$ and $\shG \in \Qcoh Z$,
we have that $\Ass_X \left(j_*\shG\right) = j(\Ass_Z \shG)$.
\end{enua}
\end{fct}

\begin{fct}[{cf.\ \cite[(3.1.6)]{EGA4-2}}]
Let $X$ be a locally noetherian scheme and $\shF\in \coh X$.
Then $\Ass \shF \cap U$ is a finite set for any quasi-compact open subset $U$ of $X$.
In particular, $\Ass \shF$ is a finite set when $X$ is noetherian.
\end{fct}

The following example will be used repeatedly in this paper.
\begin{ex}
Let $X$ be a locally noetherian scheme.
\begin{enua}
\item
If $X$ is an integral scheme with generic point $\eta$,
then $\Min X = \Ass X=\{\eta\}$ by \cref{fct:Ass loc noetherian} (1).
\item
For a point $x\in X$,
consider the reduced closed subscheme $Z_x$ of $X$ whose underlying space is $\ol{\{x\}}$.
Then $\Ass_X \shO_{Z_x}=\{x\}$ by (1) and \cref{fct:Ass loc noetherian} (5).
\end{enua}
\end{ex}

The following proposition asserts
that the property of having no embedding points is stalk-local.
\begin{prp}\label{prp:no emb pt stalk-local}
Let $X$ be a locally noetherian scheme and $\shF \in \Qcoh X$.
\begin{enua}
\item
$\Min_X(\shF)\cap \Spec\shO_{X,x} = \Min_{\shO_{X,x}}(\shF_x)$
for any point $x\in X$.
\item
The following are equivalent:
\begin{itemize}
\item
$\shF$ has no embedding points as a quasi-coherent $\shO_X$-module.
\item
$\shF_x$ has no embedding prime ideals as an $\shO_{X,x}$-module for any $x\in X$.
\end{itemize}
\end{enua}
\end{prp}
\begin{proof}
(1) follows from
\cref{fct:Ass loc noetherian} (2) and the fact that $\Spec\shO_{X,x}$ is generalization-closed.
(2) follows from (1) and \cref{fct:Ass loc noetherian} (2).
\end{proof}

The following fact is one version of the primary decomposition.
\begin{fct}[{cf.\ \cite[(3.2.6)]{EGA4-2}}]\label{fct:primary}
Let $X$ be a noetherian scheme and $\shF \in \coh X$.
Suppose that $\Ass \shF = \{x_1,\dots , x_n\}$.
Then there are surjective morphisms $\shF \surj \shF_i$ in $\coh X$ for $1\le i \le n$
satisfying the following:
\begin{itemize}
\item 
$\Ass \shF_i=\{x_i\}$ for any $1\le i \le n$.
\item
The canonical morphism $\shF \to \bigoplus_{i=1}^n \shF_i$ is injective.
\end{itemize}
\end{fct}

Let $X$ be a locally noetherian scheme of finite Krull dimension and $\shF\in \Qcoh X$.
We denote by $\Assh \shF$
the set of $x\in \Supp \shF$ such that $\dim \ol{\{x\}} = \dim \shF$.
It is easy to see that $\Assh \shF\subseteq \Min \shF$ 
by \cref{fct:Ass loc noetherian} (3).
We say that $\shF$ is \emph{equidimensional} if $\Assh \shF = \Min \shF$.
Then $X$ is equidimensional (cf.\ \cite[Definition 5.5]{GW})
if and only if $\shO_X$ is equidimensional.
We write $\Assh X := \Assh \shO_X$.
The following observation is often useful:
$\dim \shF = \dim X$ if and only if $\Assh\shF \subseteq \Assh X$.
Let $R$ be a commutative noetherian ring  of finite Krull dimension and $M\in \Mod R$.
Then we also denote by $\Assh M$ the set of $\pp \in \Supp M$
such that $\dim R/\pp = \dim M$.

Let $X$ be a locally noetherian scheme of finite Krull dimension and $\shF \in \Qcoh X$.
It may be useful to summarize this subsection as follows:
\begin{itemize}
\item
In general, $\Assh \shF \subseteq \Min \shF \subseteq \Ass \shF \subseteq \Supp \shF$ holds.
\item
$\shF$ has no embedding point if and only if $\Min \shF = \Ass \shF$.
\item
$\shF$ is equidimensional if and only if $\Assh\shF = \Min \shF$.
\item
$\shF$ is equidimensional and has no embedding points if and only if $\Assh\shF = \Ass \shF$.
\item
$\dim \shF = \dim X$ if and only if $\Assh\shF \subseteq \Assh X$.
\end{itemize}

\subsection{Torsionfree sheaves}\label{ss:torf sh}
In this subsection,
we study torsionfree sheaves on an \emph{arbitrary} noetherian scheme.
There are detailed accounts for this, for example, \cite{Ley16,EGA4-4}.
However, these references use
the notions of \emph{the torsion subsheaf} and \emph{the sheaf of meromorphic functions}.
These notions are somewhat difficult to handle. 
For example, the torsion subsheaf of a coherent sheaf and 
the sheaf of meromorphic functions on a noetherian scheme may not be quasi-coherent.
Our treatment in this subsection does not use these notions
but emphasizes the notion of associated points, which we believe to be more accessible.

We first recall the definition of torsionfree modules on an arbitrary commutative ring.
Let $R$ be a commutative ring and $M$ an $R$-module.
An element $a\in R$ is said to be \emph{$M$-regular}
if the $R$-linear map $a\colon M \to M$ given by $x \mapsto ax$ is injective.
An $R$-regular element is nothing but a nonzero divisor of $R$.
We denote by $\NZD(M)$ the set of $M$-regular elements.
\begin{dfn}
Let $R$ be a commutative ring.
An $R$-module $M$ is said to be \emph{torsionfree}
if $\NZD(R)\subseteq \NZD(M)$ holds,
that is, if any nonzero divisor of $R$ is an $M$-regular element.
\end{dfn}

\begin{ex}
We give examples of torsionfree modules.
\begin{enua}
\item
The zero module is torsionfree since $\NZD(0)=R$.
\item
Every module over a commutative artinian local ring $R$ is torsionfree.
It follows from the fact that any element of $R$ is either invertible or nilpotent.
\end{enua}
\end{ex}

A torsionfree module can be characterized using associated points.
\begin{prp}\label{prp:char torf mod}
Let $R$ be a commutative noetherian ring and $M$ an $R$-module.
\begin{enua}
\item
$M$ is torsionfree if and only if
any associated prime of $M$ is contained in some associated prime ideal of $R$.
\item
If $R$ has no associated prime ideals,
then we have that $M$ is torsionfree if and only if $\Ass M \subseteq \Ass R$.
\end{enua}
\end{prp}
\begin{proof}
(1)
By the definition, $M$ is torsionfree if and only if $\NZD(R)\subseteq \NZD(M)$.
It is equivalent to the following by \cite[Theorem 6.1 (ii)]{Matsumura}:
\[
\bigcup_{\qq \in \Ass R} \qq \supseteq \bigcup_{\pp \in \Ass M} \pp.
\]
By the prime avoidance,
it is also equivalent to that 
for any $\pp \in \Ass M$ there exists $\qq\in \Ass R$ such that $\pp \subseteq \qq$.

(2)
Any associated prime ideals of $R$ are minimal prime ideals
since $R$ has no embedded prime ideals.
Hence, if $\pp \subseteq \qq$ for some $\pp \in \Ass M$ and $\qq \in \Ass R$,
we have $\pp= \qq$.
Therefore $M$ is torsionfree if and only if $\Ass M \subseteq \Ass R$ by (1).
\end{proof}

The property of being torsionfree is not stable under localization in general
(see \cite[Example 6.7]{CS} 
and the proof of (iii)$\imply$(ii) in \cref{prp:torf sh general} below).
However, the following holds.
\begin{lem}\label{prp:torf mod local}
Let $R$ be a commutative ring and $M$ an $R$-module.
If $M_{\pp}$ is torsionfree for any $\pp\in \Spec R$,
then so is $M$.
\end{lem}
\begin{proof}
Let $a \in \NZD(R)$.
It is easy to see that $\frac{a}{1} \in \NZD(R_{\pp})$ for any $\pp \in \Spec R$.
Thus, if $M_{\pp}$ is torsionfree for any $\pp\in \Spec R$,
then $\frac{a}{1} \colon \prod_{\pp\in\Spec R} M_{\pp} \to \prod_{\pp\in\Spec R} M_{\pp}$ 
is injective.
We obtain that $a \in \NZD(M)$ by the following commutative diagram:
\[
\begin{tikzcd}
M \arr{r,"a"} \arr{d,hook} & M \arr{d,hook}\\
\prod_{\pp\in\Spec R} M_{\pp} \arr{r,,"{\frac{a}{1}}",hook} & \prod_{\pp\in\Spec R} M_{\pp}.
\end{tikzcd}
\]
\end{proof}

Before defining torsionfree sheaves,
we discuss the relationship between the conditions of quasi-coherent sheaves
related to torsionfree modules.
\begin{prp}\label{prp:torf sh general}
Let $X$ be a locally noetherian scheme and $\shF \in \Qcoh X$.
Consider the following conditions:
\begin{enur}
\item
$\shF_x$ is a torsionfree $\shO_{X,x}$-module
for any point $x\in X$.
\item
$\shF(U)$ is a torsionfree $\shO_X(U)$-module
for any affine open subset $U$ of $X$.
\item
$\Ass\shF \subseteq \Ass X$ holds.
\item
There exists an affine open covering $\{U_i\}_{i\in I}$ of $X$
such that $\shF(U_i)$ is a torsionfree $\shO_X(U_i)$-module
for any $i\in I$.
\item
Any associated point of $\shF$ is a generalization of some associated point of $X$.
\end{enur}
Then {\upshape (i)$\equi$(ii)$\equi$(iii)$\imply$(iv)$\imply$(v)} hold.
\end{prp}
\begin{proof}
The implication (i)$\imply$(ii) follows from \cref{prp:torf mod local}.
The implication (ii)$\imply$(iv) is obvious.
The implication (iii)$\imply$(i) follows from \cref{prp:char torf mod} and the following:
\[
\Ass_{\shO_{X,x}}\shF_x = \Ass \shF \cap \Spec\shO_{X,x}
\subseteq \Ass X \cap\Spec\shO_{X,x} = \Ass \shO_{X,x}.
\]
We prove (iv)$\imply$(v).
Let $x\in \Ass \shF$.
Then $x\in U_i$ for some $i\in I$.
There exists $y\in \Ass U_i$ such that $x$ is a generalization of $y$
by $\Ass \shF \cap U_i=\Ass_{\shO(U_i)}\shF(U_i)$ and \cref{prp:char torf mod}.
Thus (v) follows.

We prove (ii)$\imply$(iii).
It is enough to show that the following statement:
\begin{itemize}
\item 
Let $R$ be a commutative noetherian ring and $M\in \Mod R$.
If $M_f$ is a torsionfree $R_f$-module for any $f\in R$,
then $\Ass M \subseteq \Ass R$ holds.
\end{itemize}
We first note that $\Ass R$ is finite.
Let $\pp \in \Ass M$.
If there is no associated prime ideals $\qq$ of $R$ such that $\pp \subsetneq \qq$,
then we can easily see that $\pp \in \Ass R$
by \cref{prp:char torf mod} and the fact that $M$ is torsionfree.
Consider the case where there exists an associated prime ideal $\qq$ of $R$ 
such that $\pp \subsetneq \qq$.
Let $\qq_1,\dots, \qq_r$ be
all the associated prime ideals of $R$ such that $\pp \subsetneq \qq_i$.
Take $f \in \left(\cap_{i=1}^r \qq_i \right)\setminus \pp$.
Then $\pp \in \Ass M \cap D(f)=\Ass M_f$ and $\qq_i \not\in \Ass R \cap D(f)=\Ass R_f$ for any $i$.
Since $M_f$ is torsionfree, there exists $\qq \in \Ass R \cap D(f)$ such that $\pp \subseteq \qq$
by \cref{prp:char torf mod}.
However, from the choice of $f$, we have $\pp=\qq$.
Therefore, we obtain $\Ass M \subseteq \Ass R$.
\end{proof}

\begin{dfn}\label{dfn:torf sh}
Let $X$ be a locally noetherian scheme and $\shF \in \Qcoh X$.
\begin{enua}
\item
$\shF$ is said to be \emph{locally torsionfree}
if $\shF_x$ is a torsionfree $\shO_{X,x}$-module for any point $x\in X$.
\item
$\shF$ is said to be \emph{torsionfree}
if any associated point of $\shF$ is a generalization of some associated point of $X$.
\end{enua}
\end{dfn}

\begin{ex}
Let $R$ be a commutative noetherian ring and $M\in \Mod R$.
Consider the quasi-coherent $\shO_{\Spec R}$-module $\wti{M}$ associated to $M$.
Then $\wti{M}$ is a torsionfree $\shO_{\Spec R}$-module 
if and only $M$ is a torsionfree $R$-module.
However, even if $M$ is a torsionfree $R$-module,
$\wti{M}$ is not necessarily locally torsionfree.
\end{ex}


As the following corollary shows,
the two conditions in \cref{dfn:torf sh} coincide 
on a locally noetherian scheme without embedding points. 
\begin{cor}\label{prp:torf sh without emb pt}
Let $X$ be a locally noetherian scheme without embedding points.
The following are equivalent for $\shF \in \Qcoh X$.
\begin{enur}
\item
$\shF$ is torsionfree.
\item
$\shF$ is locally torsionfree.
\item
There exists an affine open covering $\{U_i\}_{i\in I}$ of $X$
such that $\shF(U_i)$ is a torsionfree $\shO_X(U_i)$-module
for any $i\in I$.
\end{enur}
In this case, $\shF$ has no embedding points.
\end{cor}
\begin{proof}
We have already shown (ii)$\imply$(iii)$\imply$(i) in \cref{prp:torf sh general}.
We have that $\Min X=\Ass X$ since $X$ has no embedding point.
This implies that $\Ass X$ is generalization-closed.
Thus if (i) holds, then $\Ass \shF \subseteq \Ass X$ holds.
This proves (i)$\imply$(ii) by \cref{prp:torf sh general}.

If $\shF$ is locally torsionfree, we have $\Ass \shF \subseteq \Ass X = \Min X$.
From this, every associated point of $\shF$ is maximal with respect to the specialization-order.
Hence $\shF$ has no embedding points.
This finishes the proof.
\end{proof}

\subsection{The relationship between torsionfree, pure and Cohen-Macaulay sheaves}\label{ss:torf pure CM}
In this subsection,
we describe the relationship between torsionfree, pure and Cohen-Macaulay sheaves.

We first describe the relationship between torsionfree and pure sheaves.
Let $X$ be a noetherian scheme of finite Krull dimension and $\shF\in \coh X$.
Recall that $\shF$ is said to be \emph{pure}
if for any nonzero coherent $\shO_X$-submodule $\shG$ of $\shF$,
we have that $\dim \shG=\dim\shF$.
We say that $\shF$ is \emph{maximal pure} if either $\shF=0$ or it is pure and $\dim \shF =\dim X$.

\begin{prp}[{cf.\ \cite[page 3]{Huy}}]\label{prp:pure Ass}
Let $X$ be a noetherian scheme of finite Krull dimension.
The following are equivalent for $\shF \in \coh X$.
\begin{enur}
\item
$\shF$ is pure.
\item
$\shF$ is equidimensional and has no embedding points,
that is, $\Assh \shF = \Min \shF = \Ass \shF$ holds.
\end{enur}
\end{prp}
\begin{proof}
Suppose that $\shF$ is pure.
Let $x\in \Ass \shF$ 
and $Z_x$ the reduced closed subscheme of $X$ whose underlying space is $\ol{\{x\}}$.
Then there exist a nonzero coherent ideal sheaf $\shI$ of $\shO_{Z_x}$ 
and an injective $\shO_X$-linear map $\shI \inj \shF$ by \cite[Proposition B.12]{CS}.
Note that $\Ass_X \shI = \{x\}$
since $\shI$ is a nonzero $\shO_{X}$-submodule of $\shO_{Z_x}$.
Because $\shF$ is pure,
we have that $\dim \shI = \dim \shF$.
This implies that $x \in \Assh \shF$.
Thus, we obtain $\Assh \shF = \Ass \shF$,

Conversely, suppose that $\shF$ is equidimensional and has no embedding points.
Let $\shG$ be a nonzero coherent $\shO_X$-submodule of $\shF$.
Then $\dim \shG = \dim \shF$ follows from the following:
\[
\Assh \shG \subseteq \Ass \shG \subseteq \Ass \shF = \Assh\shF.
\]
Hence $\shF$ is pure.
\end{proof}

\begin{cor}\label{prp:maximal pure Ass}
Let $X$ be a noetherian scheme of finite Krull dimension and $\shF \in \coh X$.
Then $\shF$ is maximal pure if and only if $\Ass \shF \subseteq \Assh X$.
In particular, maximal pure $\shO_X$-modules are locally torsionfree.
\end{cor}
\begin{proof}
Suppose that $\shF$ is maximal pure.
Then $\Assh \shF \subseteq \Assh X$ holds since $\dim \shF = \dim X$.
Because $\shF$ is equidimensional and has no embedding points by \cref{prp:pure Ass},
we have that $\Ass \shF =\Assh \shF \subseteq \Assh X$.

Conversely, suppose that $\Ass \shF \subseteq \Assh X$ and $\shF \ne 0$.
Then $\dim \ol{\{x\}}=\dim X$ for any $x \in \Ass \shF$.
This implies that $\dim \shF = \dim X$ and $\Ass \shF = \Assh \shF$.
\end{proof}

\begin{cor}\label{cor:pure and torf}
Let $X$ be an equidimensional noetherian scheme without embedding points.
Then the following are equivalent: 
\begin{enur}
\item
$\shF$ is maximal pure.
\item
$\shF$ is locally torsionfree.
\item
$\shF$ is torsionfree.
\end{enur}
\end{cor}
\begin{proof}
It follows from \cref{prp:torf sh without emb pt,prp:maximal pure Ass}
since $\Assh X =\Ass X$.
\end{proof}

Let us recall the definition of Cohen-Macaulay sheaves.
Let $R$ be a commutative noetherian local ring and $M\in \catmod R$.
In general, the following inequalities hold (cf.\ \cite[Proposition 1.2.12]{BH}):
\[
\depth M \le \dim M \le \dim R.
\]
We say that $M$ is (resp.\ \emph{maximal}) \emph{Cohen-Macaulay}
if either $M=0$ or $\depth M= \dim M$ (resp.\ $\depth M= \dim R$) holds.
Let $X$ be a locally noetherian scheme and $\shF \in \coh X$.
Then $\shF$ is said to be (resp.\ \emph{maximal}) \emph{Cohen-Macaulay}
if $\shF_x$ is a (resp.\ maximal) Cohen-Macaulay $\shO_{X,x}$-module for any $x\in X$.
We call $X$ \emph{Cohen-Macaulay} if $\shO_X$ is Cohen-Macaulay.
Note that $X$ is Cohen-Macaulay if and only if
every locally free $\shO_X$-module is maximal Cohen-Macaulay.

We recall some basic facts about Cohen-Macaulay modules.
\begin{fct}[{cf.\ \cite[Theorem 17.3]{Matsumura},
\cite[Tag \href{https://stacks.math.columbia.edu/tag/031Q}{031Q},
\href{https://stacks.math.columbia.edu/tag/00NT}{00NT}]{SP}}]\label{fct:CM mod basic}
Let $R$ be a commutative noetherian local ring and $M\in \catmod R$.
\begin{enua}
\item
If $M$ is Cohen-Macaulay, then it is equidimensional and has no embedding prime ideals.
More precisely, $\depth M  =\dim M = \dim R/\pp$ hold for any $\pp \in \Ass M$.
\item
If $\dim R \le 1$, then $M$ is Cohen-Macaulay
if and only if it has no embedding prime ideals.
\item
If $R$ is regular, then every maximal Cohen-Macaulay $R$-module is free.
\end{enua}
\end{fct}

The following corollary immediately follows from the fact above.
\begin{cor}\label{cor:CM sh basic}
Let $X$ be a locally noetherian scheme and $\shF \in \coh X$.
\begin{enua}
\item
If $\shF$ is Cohen-Macaulay, then it has no embedding points.
In particular, Cohen-Macaulay schemes have no embedding points.
\item
If $\dim X \le 1$,
then $\shF$ is Cohen-Macaulay if and only if it has no embedding points.
In particular,
$X$ is Cohen-Macaulay if and only if it has no embedding points
when $\dim X \le 1$
\item
If $X$ is regular,
then $\shF$ is maximal Cohen-Macaulay if and only if $\shF$ is locally free.
\qed
\end{enua}
\end{cor}
Thus,
the notion of torsionfree and locally torsionfree coincide on Cohen-Macaulay schemes
by \cref{prp:torf sh without emb pt}.

We now describe the relationship between torsionfree and Cohen-Macaulay sheaves.
\begin{prp}\label{prp:MCM and torf}
Let $X$ be a locally noetherian scheme and $\shF \in \coh X$.
If $\shF$ is maximal Cohen-Macaulay,
then it is locally torsionfree.
\end{prp}
\begin{proof}
Let $x\in \Supp \shF$.
Then $\shF_x$ is a nonzero maximal Cohen-Macaulay $\shO_{X,x}$-module.
From this,
$\shF_x$ is maximal pure by \cref{fct:CM mod basic} (1).
Thus $\shF_x$ is torsionfree by \cref{prp:maximal pure Ass}.
This means that $\shF$ is locally torsionfree.
\end{proof}

\begin{rmk}
Let $X$ be a noetherian scheme of finite Krull dimension.
\begin{enua}
\item
If $X$ is equidimensional and has no embedding points,
then maximal Cohen-Macaulay $\shO_X$-modules are maximal pure
by \cref{cor:pure and torf,prp:MCM and torf}.
The converse does not hold.
Consider the formal power series ring $R=\bbk[[x,y]]$ of two variables over a field $\bbk$.
Then $R$ is equidimensional and has no embedding prime ideals
since it is an integral domain.
Its maximal ideal $\mm=(x,y)$ is a maximal pure but not Cohen-Macaulay $R$-module.
Indeed, we have $\Ass \mm = \{(0)\}$ by $\emptyset \ne \Ass \mm \subseteq \Ass R =\{(0)\}$.
Thus $\mm$ is maximal pure.
On the other hand, we have $\depth \mm =1<2=\dim \mm$ 
by the exact sequence $0\to \mm \to R \to \bbk \to 0$ of $R$-modules.
Hence $\mm$ is not Cohen-Macaulay.

\item
If $X$ is not equidimensional,
there is a maximal Cohen-Macaulay $\shO_X$-module which is not pure.
Indeed, consider $R=\bbk\times \bbk[x]$, where $\bbk$ is a field.
Then $R$ itself is a maximal Cohen-Macaulay but not equidimensional $R$-module.
\end{enua}
\end{rmk}

\begin{lem}\label{prp:torf + CM = MCM}
Let $X$ be a Cohen-Macaulay scheme and $\shF\in\coh X$.
\begin{enua}
\item
If $\shF$ is torsionfree, then $\dim \shF_x = \dim \shO_{X,x}$ for any $x \in X$.
\item
$\shF$ is maximal Cohen-Macaulay
if and only if it is torsionfree and Cohen-Macaulay.
\end{enua}
\end{lem}
\begin{proof}
We only prove (1) since (2) is the direct consequence of (1) and \cref{prp:MCM and torf}.
Since $X$ is Cohen-Macaulay, 
the local ring $\shO_{X,x}$ is equidimensional and has no embedding points for any $x\in X$
by \cref{fct:CM mod basic}.
Therefore the torsionfree module $\shF_x$ is maximal pure by \cref{cor:pure and torf}.
In particular, we have $\dim \shF_x = \dim \shO_{X,x}$.
\end{proof}

\begin{cor}\label{prp:torf on 1-dim CM}
Let $X$ be a Cohen-Macaulay scheme such that $\dim X \le 1$.
The following are equivalent for $\shF \in \coh X$.
\begin{enur}
\item
$\shF$ is maximal Cohen-Macaulay.
\item
$\shF$ is locally torsionfree.
\item
$\shF$ is torsionfree.
\end{enur}
\end{cor}
\begin{proof}
The implications (i)$\imply$(ii)$\equi$(iii) follow from 
\cref{prp:torf sh without emb pt,prp:MCM and torf} since $X$ has no embedding points.
We prove (ii)$\imply$(i).
Suppose that $\shF$ is locally torsionfree.
Then it has no embedding points by \cref{prp:torf sh without emb pt},
and hence it is Cohen-Macaulay by \cref{cor:CM sh basic}.
Thus $\shF$ is maximal Cohen-Macaulay by \cref{prp:torf + CM = MCM}.
\end{proof}

\begin{cor}\label{prp:torf on regular}
Let $X$ be a regular scheme such that $\dim X \le 1$.
The following are equivalent for $\shF \in \coh X$.
\begin{enua}
\item
$\shF$ is locally free.
\item
$\shF$ is maximal Cohen-Macaulay.
\item
$\shF$ is locally torsionfree.
\item
$\shF$ is torsionfree.
\end{enua}
\end{cor}
\begin{proof}
It follows from \cref{cor:CM sh basic,prp:torf on 1-dim CM}.
\end{proof}

Let $X$ be a noetherian scheme of finite Krull dimension.
The relationship among the properties of coherent $\shO_X$-modules 
discussed in this subsection is summarized as follows:
\[
\begin{tikzcd}
& \text{without embdding points} & \text{locally free} \arr{d,Rightarrow,bend left=20,"\text{$X$: Cohen-Macaulay}"} \\
\text{maximal pure} \arr{rd,Rightarrow} \arr{ru,Rightarrow} &  & \text{maximal Cohen-Macaulay} \arr{lu,Rightarrow} \arr{ld,Rightarrow} \arr{u,Rightarrow,bend left=20,"\text{$X$: regular}"}\\
& \text{locally torisonfree} \arr{uu,"\text{$\Min X = \Ass X$}",Rightarrow,bend left=15} \arr{d,Rightarrow} \arr{lu,Rightarrow,bend left=10,"\Assh X=\Ass X"} \arr{ru,Rightarrow,bend right=10,"\text{$X$: Cohen-Macaulay and $\dim X \le 1$}"'} & \\
& \text{torisonfree} \arr{u,Rightarrow,bend left=25,"\Min X= \Ass X"}&.
\end{tikzcd}
\]


\end{document}